\documentclass{birkjour}
 \newtheorem{thm}{Theorem}[section]
 \newtheorem{cor}[thm]{Corollary}
 
 \newtheorem{lem}[thm]{Lemma}
 
 \newtheorem{prop}[thm]{Proposition}
 \theoremstyle{definition}
 \newtheorem{defn}[thm]{Definition}
 \theoremstyle{remark}
 \newtheorem{rem}[thm]{Remark}
 
 \numberwithin{equation}{section}
 \usepackage{hyperref}
\usepackage{color}
\usepackage{bbm}
\def\multiset#1#2{\ensuremath{\left(\kern-.3em\left(\genfrac{}{}
{0pt}{}{#1}{#2}\right)\kern-.3em\right)}}
\begin{document}

%
%
%
%
%
%
%
%
%
\begin{abstract}
We investigate Sperner's labelings of $H^\pi_{k,q}$, the
hypergraph whose hyperedges are facets of the edgewise
triangulation of a $(k-1)$-simplex defined by a permutation
$\pi\in \mathbb{S}_{k-1}$. Mirzakhani and Vondr\' ak showed that
the greedy coloring of $H^{\mathrm{Id}}_{k,q}$ produces the
maximal number of monochromatic hyperedges. The line graph of
$H_{k,q}^\pi$ is built from the copies of the graph $G_\pi$ that
represents which subsets of consecutive numbers of $[k-1]$ are
contiguous in $\pi$. We characterize these graphs in terms of
dissections a regular $k$-gon and also show how they encode the
adjacency relation between a hypersimplex and the facets of its
alcoved triangulation. The natural action of the dihedral group
$D_k$ on a regular $k$-gon and graphs $G_\pi$ extends on the group
of permutations $\mathbb S_{k-1}$. Independent sets of the graphs
$G_\pi$ of the permutations that are not invariant under the
rotation are used to define a class of Sperner's colorings that
produce more monochromatic hyperedges then the greedy colorings.
This colorings are also optimal for a certain permutations.

\end{abstract}

\title[]
 {Sperner's colorings of hypergraphs
arising from edgewise triangulations}

\author[Du\v sko Joji\' c]{Du\v sko Joji\' c}

\address{%
University of Banja Luka, Faculty of Science\\
Mladena Stojanovi\' ca 2, 78 000 Banja Luka\\
Bosnia and Herzegovina}

\email{dusko.jojic@pmf.unibl.org}
\author[Ognjen Papaz]{Ognjen Papaz}

\address{%
University of East Sarajevo, Faculty of Philosophy\\
Alekse \v Santi\'ca 1, 71 420 Pale\\
Bosnia and Herzegovina}

\email{ognjen.papaz@ff.ues.rs.ba}



\date{}
\dedicatory{}


\maketitle

\section{Introduction}
\subsection{Edgewise subdivision, hypergraphs, Sperner's labeling}

The edgewise subdivision $T_{k,q}$ of a $(k-1)$-dimensional
simplex $\Delta$ is a specific triangulation that divides each
$d$-dimensional face of $\Delta$ into $q^d$ smaller simplices.
This triangulation, defined in \cite{Edel}, has many applications,
one of which is being a framework for the Hypergraph Labeling
Problem considered by Mirzakhani and Vondrak in \cite{M-V} and
\cite{M-VforM}. We briefly revise the notation defined in
\cite{M-V}. For any $k,q\in \mathbb{N}$ let
$$R_{k,q}=\left\{\mathbf{x}=(x_1,x_2,\ldots,x_{k-1})\in
\mathbb{R}^{k-1}:0\leq x_1\leq x_2\leq \cdots\leq x_{k-1}\leq
q\right\}.$$ The
vertices of this $(k-1)$-simplex are
$$\mathbf{w}_i=(\underbrace{0,\ldots,0}_{k-i},
\underbrace{q,\ldots,q}_{i-1})\textrm{ for all }i\in[k]\textrm{
and } R_{k,q}=
\operatorname{conv}\left\{\mathbf{w}_1,\mathbf{w}_2,\ldots,\mathbf{w}_k\right\}.$$
\noindent The \textit{edgewise subdivision} $T_{k,q}$ of $R_{k,q}$
is a triangulation of $R_{k,q}$ whose vertex set is
$$W_{k,q}=\{\mathbf{v}\in \mathbb{Z}^{k-1}:
0\leq v_1\leq v_2\leq \cdots\leq v_{k-1}\leq q\}.$$
The maximal cells of $T_{k,q}$ are described in the following manner.
A permutation $\pi=\pi_1 \pi_2\ldots\pi_{k-1}\in
 \mathbb{S}_{k-1}$ is \textit{consistent}
 with a vertex $\mathbf{v}\in W_{k,q-1}$ if $i$
 appears before
 $i+1$ in $\pi$
 whenever $v_i=v_{i+1}$.
For a vertex $\mathbf{v}\in W_{k,q-1}$, and a permutation $\pi \in
\mathbb{S}_{k-1}$ that is consistent with $\mathbf{v}$, we let
$F(\mathbf{v},\pi)$ denote $(k-1)$-dimensional simplex
$F(\mathbf{v},\pi)=\operatorname{conv}\left\{\mathbf{v}^{(1)},
\mathbf{v}^{(2)},\ldots,\mathbf{v}^{(k)}\right\}$, where
\begin{equation}\label{E:tjemenahiperivice}
\mathbf{v}^{(1)}=\mathbf{v},
 \mathbf{v}^{(2)}=\mathbf{v}+e_{\pi_{k-1}},
 \ldots\textrm{, }\mathbf{v}^{(k)}=\mathbf{v}+
 e_{\pi_{k-1}}+\cdots+e_{\pi_{1}}=\mathbf{v}+\mathbbm{1}.
\end{equation}
The maximal cells or \textit{facets}
 of the edgewise triangulation $T_{k,q}$ are all
 simplices $ F(\mathbf{v},\pi)$ where $\mathbf{v}\in W_{k,q-1}$, and
 $\pi\in \mathbb{S}_{k-1}$ is consistent with $\mathbf{v}$.
 We say that $F(\mathbf{v},\pi)$ is \emph{a facet of
 type} $\pi$. Informally speaking, $\mathbf{v}$ is the first vertex of
$F(\mathbf{v},\pi)$, and the remaining vertices of this facet are
obtained by increasing coordinates
 of $\mathbf{v}$ one by one,
 reading $\pi$ backwards. This interpretation also gives us
 an ordering of the vertices of $F(\mathbf{v},\pi)$.

Assume that a simplicial complex $K$ is a triangulation of a
$(k-1)$-simplex $\Delta^{k-1}$. A \textit{Sperner labeling} of $K$
is a map $\ell:V(K)\rightarrow [k]$ that assigns different labels
to the vertices lying on the complementary faces of
$\Delta^{k-1}$.

For the edgewise subdivision of $R_{k,q}$, we assign the labels
for the vertices of $R_{k,q}$ by setting
$\ell(\mathbf{w}_i)=k+1-i$. If we assume that $v_0=0$ and $v_k=q$,
then the list of admissible labels for a vertex
$\mathbf{v}=(v_1,v_2,\ldots,v_{k-1})\in W_{k,q}$ in a Sperner's
labeling of $T_{k,q}$ is $L(\textbf{v})=\{i:v_{i-1}<v_i\}\subset
[k]$. Note that a map $\ell: W_{k,q}\rightarrow [k]$ is a
Sperner's labeling of $T_{k,q}$ if and only if
\begin{equation}\label{Spernerov uslov}
\ell(\textbf{v})\in L(\textbf{v}),\ \text{for all}\ \textbf{v}\in
W_{k,q}.
\end{equation}

 \begin{defn} For a fixed $\pi\in \mathbb{S}_{k-1}$, consider
 the $k$-regular hypergraph $H^{\pi}_{k,q}$ whose hyperedegs are
 all facets of $T_{k,q}$ of type $\pi$, i.e.,
 $$E(H^\pi_{k,q})=\{F(\mathbf{v},\pi):\mathbf{v}\in W_{k,q-1},
  \pi\textrm{ is consistent with }\mathbf{v}\}.$$
Note that the vertex set $V(H_{k,q}^\pi)\subset W^\pi_{k,q}$ is
the union of all sets $F(\mathbf{v},\pi)$. The hypergraph
$H_{k,q}^\pi$ is called a \emph{simplex-lattice hypergraph}.

A map $\ell: V(H_{k,q}^\pi)\to[k]$ is called a \emph{Sperner's labeling
of} $H_{k,q}^\pi$ if (\ref{Spernerov uslov}) holds. For a given
Sperner labeling of $H_{k,q}^\pi$, a hyperedge
$F(\mathbf{v}, \pi)$ is called \emph{monochromatic} if all
vertices in $F(\mathbf{v}, \pi)$ have the same label.
\end{defn}

\textbf{Problem 1.} For a given $\pi\in \mathbb{S}_{k-1}$, find a
Sperner labeling of $H^\pi_{k,q}$ that minimizes the number of
non-monochromatic hyperedges. Such a labeling will be called an
\emph{optimal Sperner's labeling} of $H_{k,q}^\pi$.

An instance of Hypergraph Labeling Problem that is discussed in
\cite{M-V} seeks for an optimal Sperner's labeling of the
simplex-lattice hypergraph $H_{k,q}^{\mathrm{Id}}$.

\begin{prop}\label{greedy} (Proposition 2.1. in \cite{M-V})
The greedy coloring (first-choice labeling) $\ell(v)=min
L(\mathbf{v})$ is an optimal Sperner's labeling of the
simplex-lattice hypergraph $H_{k,q}=H_{k,q}^{\mathrm{Id}}$. The
number of non-monochromatic cells in the greedy coloring of
$H_{k,q}^{\mathrm{Id}}$ is $|W_{k,q-1}|-|W_{k,q-2}|={q+k-3\choose
k-2}$.
\end{prop}
 A permutation $\pi\in \mathbb{S}_{k-1}$ is consistent
with a vertex $\mathbf{v}=(v_1,\ldots,v_{k-1})\in W_{k,q-1}$ if
and only if $v_i<v_{i+1}$ whenever $i+1$ precedes $i$ in $\pi$.
This notion of consistency between vertices of $W_{k,q-1}$ and the
permutations from $\mathbb{S}_{k-1}$ motivates the following
definition.
\begin{defn}
  For a permutation $\pi=\pi_1\pi_2\ldots\pi_{k-1}\in
\mathbb{S}_{k-1}$ we define \emph{the set of adjacent inversions}
of $\pi$ as $S(\pi)=\big\{i\in[k-2]:i+1 \textrm{ appears before
}i\textrm{ in }\pi\big\}.$
\end{defn} Therefore, a fixed permutation $\pi$ is consistent with
$\mathbf{v}=(v_1,v_2,\ldots,v_{k-1})$ from $W_{k,q-1}$
($\mathbf{v}$ can be the first vertex for a facet of type $\pi$)
if and only if
\begin{equation}\label{E:prvotjeme}
v_i< v_{i+1}\textrm{ for all }i \in S(\pi).
\end{equation}

\begin{prop}\label{P:numberofhyperedges} The number of
hyperedges in $H^\pi_{k,q}$ is ${k+q-2-|S(\pi)|\choose k-1}$.
\end{prop}
\begin{proof}
 For all
$i\in [k-1]$, we apply the transformation $v_i\mapsto v_i-|\{j\in
S(\pi):j<i\}|,$ which effectively removes all the strict
inequalities among the coordinates of $\mathbf{v}$ imposed by
condition (\ref{E:prvotjeme}). After this transformation, the
resulting $(k-1)$-tuple has weakly increasing coordinates that lie
between $0$ and $q-|S(\pi)|-1$.
\end{proof}
\begin{rem}\label{R:blocksinpi} A \textit{block} in a
permutation $\pi\in\mathbb{S}_{k-1}$
is a maximal subsequence of $\pi$ consisting of consecutive
integers from $[k-1]$. If $S(\pi)=\{a_1,a_2,\ldots,a_{m-1}\}$,
then
 $\pi$ has $m$ blocks
$$B_1=\{1,2,\ldots,a_1\},\ldots,
B_{i}=\{a_{i-1}+1 ,\ldots,a_i\},\ldots, B_{m}=\{a_{m-1}+1
,\ldots,k-1\}.$$ Note that the first and the last element of each
block are determined by $S(\pi)$.
\end{rem}
The next statement corresponds to Problem 1.6. in \cite{Pet}.

\begin{prop}\label{P:adjinv=desc}
The set of adjacent inversions of a permutation
$\pi\in\mathbb{S}_{k-1}$ coincides with with the descent set of
its inverse $\pi^{-1}=\pi_1^{-1}\pi_2^{-1}\ldots\pi^{-1}_{k-1}$,
i.e.:
$$ S(\pi)=D(\pi^{-1})=
\big\{\, i:\pi^{-1}_i>\pi^{-1}_{i+1}\big\} \textrm{ for all }
\pi\in\mathbb{S}_{k-1}.$$\end{prop}

Consequently, the well-known statistic for the descents of
permutations
$$\beta_T(B_{k-1})=|\{\pi\in \mathbb{S}_{k-1}:S(\pi)=T\}|=
|\{\pi\in \mathbb{S}_{k-1}:D(\pi)=T\}|$$ also counts permutations
with a prescribed set of adjacent inversions. Geometrically,
$\big(\beta_T(B_{k-1})\big)_{S\subset[k-2]}$ corresponds to the
flag-$h$-vector of a $(k-2)$-dimensional simplex, refining the
$h$-vector of barycentric subdivision, see in
\cite{Bjolex},\cite{Stanleybook} and \cite{Tools} for more
details. Moreover, these numbers refine the classical Eulerian
numbers:
$$A(k-1,m)=|\{\pi\in
\mathbb{S}_{k-1}:|D(\pi)|=m\}|=\sum_{ |T |=m}\beta_T(B_{k-1}).$$

\subsection{Sorted families and triangulation of a hypersimplex}

\noindent An ordered pair
$(\{a_1,a_2,\ldots,a_m\},\{b_1,b_2,\ldots,b_m\})$, of $m$-element
subsets of $[k]$, is said to be \emph{sorted} if $a_1\leq b_1\leq
a_2\leq b_2 \leq\ldots\leq a_m\leq b_m$. An ordered family
$(S_1,S_2,\ldots,S_n)$ of $m$-element subsets of $[k]$ is said to
be \emph{sorted} if each pair $(S_i,S_j)$ is sorted for all
$i,j\in[n],i<j$. The concept of sorted families was used by
Sturmfels in \cite{Sturmfels} to describe a triangulation of a
hypersimplex.

Assume that for a permutation $\pi=\pi_1\pi_2\ldots \pi_{k-1}\in
\mathbb{S}_{k-1}$, we have $S(\pi)=\{a_1,a_2,\ldots,a_{m-1}\}$. We
define the ordered family $\mathcal S_\pi=(S_1,S_2,\ldots,S_k)$ as
follows:
\begin{equation}\label{sortirana familija}
S_1=\{1,a_1+1,\ldots,a_{m-1}+1\},\\
 \text{ and }
S_i=\big (S_{i-1}\setminus \{\pi_{i-1}\}\big)\cup\{\pi_{i-1}+1\},
\end{equation}
for all $i=2,3,\ldots,k$. \noindent The family $\mathcal S_\pi$ is
a maximal sorted family of $m$-element subsets of $[k]$. The map
$\pi\mapsto \mathcal S_\pi$ defined above is a bijection between
permutations $\pi\in \mathbb{S}_{k-1}$ with exactly $m-1$ adjacent
inversions and maximal sorted families of $m$-element subsets of
$[k]$. The inverse map is defined by $\mathcal S\mapsto
\pi_1\pi_2\ldots\pi_{k-1},$
 where $\{\pi_i\}=S_{i}\setminus
S_{i+1}$. More details can be found in \cite{Lam}.
\begin{rem}\label{R:rasporedusorted}
Let $(S_1,S_2,\ldots,S_k)$ be a maximal sorted family defined by a
permutation $\pi\in \mathbb{S}_{k-1}$. The distribution of
elements of $[k]$ among the sets $S_i$ is described as follows:
\begin{itemize}
    \item $1\in S_j$ if and only if $1\leq j\leq\pi^{-1}(1)$.
    \item If $i-1\notin S(\pi)$, then $i\in S_j$ if and
    only if $\pi^{-1}(i-1)+1\leq j\leq \pi^{-1}(i)$.
    \item If $i-1\in S(\pi)$, then $i\in S_j$ for all
     $j\leq \pi^{-1}(i)\}$ or
     $j\geq \pi^{-1}(i-1)+1$.
    \item $k\in S_j$ if and only if $\pi^{-1}(k-1)+1\leq j\leq k$.
\end{itemize}
\end{rem}

Recall that the \emph{hypersimplex} $\Delta_{k,m}$ is a
$(k-1)$-dimensional polytope obtained as the intersection of
$k$-dimensional cube $[0,1]^k$ with the hyperplane
$x_1+x_2+\cdots+x_{k}=m$. The vertices of $\Delta_{k,m}$ are
$$\left\{e_S: S\in {[k]\choose m}\right\}\textrm{, where }e_S=
\sum_{i\in S}e_i.$$ The segment $[e_S,e_T]$ is an edge of
$\Delta_{k,m}$ if and only if $|S \triangle T|=2$. For $1<m<k$,
the hypersimplex $\Delta_{k,m}$ has $2k$ facets, given by:
\begin{equation}\label{E:facets}\begin{split}
F_i=\{x\in \Delta_{k,m}:x_i=0\}\cong \Delta_{k-1,m}\textrm{,
}\\
H_i=\{x\in \Delta_{k,m}:x_i=1\}\cong \Delta_{k-1,m-1}\ (i\in[k]).
\end{split}
\end{equation}

\noindent For each $\pi\in\mathbb S_{k-1}$ with $|S(\pi)|=m-1$, we
define the simplex
$\sigma_\pi=\operatorname{conv}\left\{e_{S_1},e_{S_2},\ldots,e_{S_k}\right\}\subset
\Delta_{k,m},$ using the associated maximal sorted family
 $\mathcal{S}_\pi$ described in (\ref{sortirana familija}).
 The seminal paper
\cite{Lam} shows that some several seemingly different
triangulations of a hypersimplex are the same. In particular, this
triangulation can be described in the following way.
\begin{thm}[Theorem 2.7. in \cite{Lam}]\label{T:triangulationofHip}
For integers $k-1\geq m\geq 1$, the hypersimplex $\Delta_{k,m}$
admits a triangulation
  parametrized by
  permutations $\pi \in \mathbb{S}_{k-1}$ with $|S(\pi)| = m - 1$.
That is,
$$\Delta_{k,m}=\bigcup_{\substack{\pi\in
\mathbb{S}_{k-1}\\|S(\pi)|=m-1}}\sigma_\pi.$$
\end{thm}

Note that $\Delta_{k,m}$ is triangulated into $A(k-1,m-1)$
simplices.

\section{The line graph of $H^\pi_{k,q}$}

The vertices of a hyperedge $F=F(\mathbf{v},\pi)$ are listed in
$(\ref{E:tjemenahiperivice})$. This hyperedge of $H_{k,q}^\pi$ can
be $i$-monochromatic with respect to some Sperner's labeling of
$H_{k,q}^\pi$ if and only if $v_{i-1}<v_{i}$ for all $i\notin
S(\pi)$, and $v_{i-1}<v_{i}-1$ for all $i\in S(\pi)$. Here we
assume that $v_0=0$ and $v_k=q$. Two hyperedges of $H^\pi_{k,q}$
that share a common vertex can both be monochromatic if and only
if all their vertices are colored with the same color. This
motivates us to investigate the line graph of $H^\pi_{k,q}$, where
the hyperedges of $H^\pi_{k,q}$ (cells $F(\mathbf v,\pi)$)
 are identified with its initial vertex $\mathbf{v}$.

\begin{defn}
We let $G^{\pi}_{k,q}$ denote the \emph{line graph} of
$H^\pi_{k,q}$.
\end{defn}

The vertices of $G^{\pi}_{k,q}$ correspond to the hyperedges
(facets) of $H^\pi_{k,q}$. After identifying each cell
$F(\mathbf{v},\pi)$ with its initial vertex $\mathbf{v}$, the
vertex set of $G^{\pi}_{k,q}$ is given by
$W^\pi_{k,q}=V(G^{\pi}_{k,q})= \big\{\mathbf{v}\in
W_{k,q-1}:v_{i+1}>v_i \textrm{ for all }i\in S(\pi)\big\}.$

From Proposition \ref{P:numberofhyperedges}, it follows that the
number of vertices in $G^{\pi}_{k,q}$ is
$$|W^\pi_{k,q}|=|E(H^\pi_{k,q})|={k+q-2-|S(\pi)|\choose k-1}.$$
Two vertices in $G^{\pi}_{k,q}$ are adjacent if and only if the
corresponding facets share a common vertex. From
(\ref{E:tjemenahiperivice}), it follows that $\mathbf{xy}\
(\mathbf{x},\mathbf{y}\in W^\pi_{k,q}$) is an
 edge in
 $G^{\pi}_{k,q}$ if and only if
\begin{equation}\label{ivica}
\mathbf{y}=\mathbf{x}+e_{\pi_a}+
e_{\pi_{a}+1}+\cdots+e_{\pi_{b-1}}, \textrm{ for some }1\leq a<b
\leq k.
\end{equation} In other words, one vertex is obtained
from the other by increasing some of its coordinates whose indices
are consecutive entries of $\pi$. For each pair $1\leq a<b \leq
k$, define the set $E_{a,b}=\{\mathbf{x}\mathbf{y}\in
E(G^\pi_{k,q}):\mathbf{y}=
\mathbf{x}+e_{\pi_a}+\cdots+e_{\pi_{b-1}}\}$. Therefore, the edge
set of $G^\pi_{k,q}$ decomposes as
\begin{equation}\label{E:edgesss}
E(G^\pi_{k,q})=\bigsqcup_{1\leq a<b \leq k}E_{a,b}.
\end{equation}

\begin{prop}\label{P:edgesconsec} For all $i\in[k-1]$ we have that
$$|E_{i,i+1}|=|E_{1,k}|={k+q-3-|S(\pi)|\choose k-1}$$
\end{prop}
\begin{proof} Assume that $\mathbf{xy}$ is an edge from $E_{i,i+1}$.
In addition to the conditions described in (\ref{E:prvotjeme}),
the vertex $\mathbf{x}$ must satisfy one additional condition:
$$x_i<x_{i+1}\textrm{ if }i \notin S(\pi)
\textrm{,  or } x_i<x_{i+1}-1\textrm{ if }i \in S(\pi).$$
Similarly, if $\mathbf{xy}\in E_{1,k}$, the vertex $\mathbf{x}$
must additionally satisfy $x_{k-1}\leq q-2$. After applying the
transformation that eliminates all strong inequalities as in the
proof of Proposition~\ref{P:numberofhyperedges}, we are left with
choosing $k-1$ weakly increasing coordinates for $\mathbf{x}$,
each lying between $0$ and $q-|S(\pi)|-2$.
\end{proof}

\noindent To enumerate the remaining edges of $G^{\pi}_{k,q}$, we
need to take a closer look at how the sets of consecutive entries
of $\pi$ decompose into unions of subsets of consecutive numbers.
For every $T\subset [k-1]$ we are looking for a decomposition of the
form: $T=T_1\cup T_2\cup\cdots\cup T_c,$ where
$T_i=\{a_i,a_i+1,\ldots,b_i\}$ is a maximal subset of consecutive
numbers contained in $T$. In that case, we say that $T$ decompose
into $c$ subsets of consecutive numbers.

\begin{lem}\label{L:Eab} Assume that for a given permutation
$\pi=\pi_1\pi_2\ldots\pi_{k-1}\in \mathbb{S}_{k-1}$ and $1\leq
a<b\leq k$, the set $T_{a,b}=\{\pi_a,\pi_{a+1},\ldots,\pi_{b-1}\}$
has a decomposition into $c$ subsets of consecutive numbers. Then
the number of edges of $G^\pi_{k,q}$ contained in
 $E_{a,b}$ is
 $$|E_{a,b}|={k+q-c-2-|S(\pi)|\choose k-1}.$$
\end{lem}
\begin{proof}
Assume that $T_{a,b}=\{\pi_a, \pi_{a+1},\ldots,\pi_{b-1}\}=T_1\cup
T_2\cup\cdots\cup T_c$, where each $T_i = \{a_i,
a_i+1,\ldots,b_i\}$ is a maximal subset of consecutive integers.
 From (\ref{E:prvotjeme}), we see
that $\mathbf{xy} \in E_{a,b}$ if and only if $\mathbf{x}$
satisfies the following $c$ additional conditions:
$$ x_{b_i}<x_{b_i+1}\textrm{ if }b_i \notin S(\pi)
\textrm{,  or } x_{b_i}<x_{b_i+1}-1\textrm{ if }b_i \in S(\pi)
\textrm{, for all }i=1,\ldots,c.$$ By reasoning analogous to that
in Proposition \ref{P:edgesconsec}, we must choose $k-1$ weakly
increasing coordinates for $\mathbf{x}$ from $q-|S(\pi)|-c$ values
(one additional constraint per  $T_{i}$).
\end{proof}
Therefore, the total number of edges in $G_{k,q}^\pi$ depends on
the decomposition of consecutive entries of $\pi$, which motivates
the following definition.
\begin{defn}
For a permutation $\pi=\pi_1\pi_2\ldots\pi_{k-1}\in
\mathbb{S}_{k-1}$, let $cs_i(\pi)$ denote the number of
subsequences $T_{a,b}=\{\pi_a, \pi_{a+1},\ldots,\pi_{b-1}\}$ of
$\pi$ that decompose into exactly $i$ subsets of consecutive
integers. We define \emph{the vector of consecutiveness} of $\pi$
as
$$cs(\pi)=\left(cs_1(\pi),cs_2(\pi),\ldots,
cs_{m}(\pi)\right).$$
\end{defn}
Recall that a permutation $\pi \in \mathbb{S}_{k-1}$ is called
\emph{simple} if no set of consecutive integers having more then
one and less then $k-1$ elements is contiguous in $\pi$. Equivalently,
$\pi\in \mathbb{S}_{k-1}$ is a simple if and only if
$cs_1(\pi)=k$. For more details on simple permutations, see
\cite{Bri} or \cite{AA}.
\begin{rem} Some basic properties of the vector of consecutiveness.
\begin{enumerate}
    \item For all $\pi\in \mathbb{S}_{k-1}$,
     we have $cs(\pi)=cs(\pi^{\mathrm{op}})$, where $\pi^{\mathrm{op}}=
    \pi_{k-1}\ldots\pi_2\pi_1$.
    \item Note that $cs_i(\pi)=0$ for all
$i>\left\lceil\frac{k-1}{2}\right\rceil$, and that $\sum_{i}
cs_i(\pi)={k \choose 2}$.
    \item For all $\pi\in \mathbb{S}_{k-1}$ and
    $i\leq\left\lceil\frac{k-1}{2}\right\rceil$
    the following bounds hold \begin{equation}\label{E:boundsforcs}
ik\leq cs_1(\pi)+cs_2(\pi)+\cdots+cs_i(\pi)\leq {k\choose 2}.
\end{equation}
\end{enumerate}
We describe the permutations attaining these bounds in Section 3.
\end{rem}
\noindent Combining (\ref{E:edgesss}) and Lemma \ref{L:Eab}, we
obtain the following expression for the total number of edges in
$G_{k,q}^\pi$:
\begin{thm}
The number of edges in $G_{k,q}^\pi$ is
$$\sum_{i} cs_i(\pi){q+k-2-|S(\pi)|-i\choose k-1}$$
\end{thm}

To utilize the graph $G_{k,q}^\pi$  for finding an optimal Sperner
labeling of $H_{k,q}^\pi$, we need to describe how a Sperner's
labeling of $H_{k,q}^\pi$ transfers to a labeling of
$G_{k,q}^\pi$, and vice versa. Given a Sperner-labeled hypergraph
 $H^\pi_{k,q}$, each monochromatic hyperedge $F(\mathbf{v},\pi)$
  naturally assigns its color
  to the corresponding vertex $\mathbf{v}$ of $G_{k,q}^\pi$.
 The non-monochromatic hyperedges of
 $H^\pi_{k,q}$ leave the corresponding vertices of $G_{k,q}^\pi$
 non-colored, and we say that these vertices receive \mbox{the color $0$.}

We denote by $L^\pi(\mathbf{v})$ the list of the possible labels
for $\mathbf{v}\in W_{k,q}^\pi$. Given a Sperner-labeled hypegraph
$H_{k,q}^\pi$, the color assignment described before induces a
function $f:W_{k,q}^\pi\to[k]\cup\{0\}$ that satisfies the following
conditions:
\begin{equation}\label{uslov1}
    f(\mathbf{x})\in L^\pi(\mathbf{x})\cup\{0\}
     \text{\ for every\ } \mathbf{x}\in W^\pi_{k,q-1};
\end{equation}
\begin{equation}\label{uslov2}
    \text{\ If\ } \mathbf{xy} \text{\ is an edge in\ } G^\pi_{k,q},
    \text{\ then\ } f(\mathbf{x})=f(\mathbf{y}) \text{\ or\ }
    0\in \{f(\mathbf{x}),f(\mathbf{y})\}.
\end{equation}

On the other hand, if $f:W_{k,q}^\pi\to[k]\cup\{0\}$ satisfies
the conditions (\ref{uslov1}) and (\ref{uslov2})
then there is a Sperner's labeling of $H_{k,q}^\pi$ such that
a hyperedge $F(\mathbf{v},\pi)$  is $i$-monochromatic
if and only if $f(\mathbf{v})=i$.

A mapping $f:W_{k,q}^\pi\to[k]\cup\{0\}$ that satisfies conditions
(\ref{uslov1}) and (\ref{uslov2}) will be called a (proper)
\emph{Sperner's labeling} of $G_{k,q}^\pi$. Now, we can rephrase
Problem 1 in the following way.\vspace{.2cm}

\textbf{Problem 1a.} For a given $\pi\in\mathbb S_{k-1}$,
find a Sperner's labeling of $G_{k,q}^\pi$ that minimizes
the number vertices assigned the color $0$.\\

\noindent Assume that the set of adjacent inversions for a given
permutation $\pi\in \mathbb{S}_{k-1}$ is
$S(\pi)=\{a_1,a_2,\ldots,a_{m-1}\}$. The vertex set of
$G^\pi_{k,q}$ can be identified with the integer points in a
modified simplex $R^\pi_{k,q}\subseteq R_{k,q}$. More precisely,
define
$R^\pi_{k,q}=\operatorname{conv}\left\{\mathbf{w}_\pi^{(1)},\mathbf{w}_\pi^{(2)},\ldots,
\mathbf{w}_\pi^{(k)}\right\}$, where the vertices are given as
follows:
$$\mathbf{w}_\pi^{(1)}=(\underbrace{0,\ldots,0}_{a_1},
\underbrace{1,\ldots,1}_{a_2-a_1},\ldots,
\underbrace{m-1,\ldots,m-1}_{k-1-a_{m-1}})\textrm{, }$$
$$\textrm{ and }  \mathbf{w}_\pi^{(i+1)}=\mathbf{w}_\pi^{(i)}+
(q-1-|S(\pi)|)e_{k-i}
\textrm{ for } i=1,2,\ldots,k-1.$$

\noindent Now, the vertex $\mathbf{w}_\pi^{(i)}$ of $R^\pi_{k,q}$
can be assigned only the color $k+1-i$, since
$F(\mathbf{w}_\pi^{(i)},\pi)$
 can only be $(k+1-i)$-monochromatic.
Furthermore, a vertex of $G^\pi_{k,q}$ lying in the
relative interior of a face of $R_{k,q}^\pi$ can be colored with some color
only if some vertex of that face can be colored with the same color.

For a given vertex $\mathbf{v} \in W_{k,q}^\pi$ the set of
admissible colors with respect to a Sperner's labeling is:
\begin{equation}\label{listaboja}
L^\pi(\mathbf{v})=\{i\in S(\pi):v_{i-1}<v_{i}+1\}\cup \{i\in
[k]\setminus S(\pi):v_{i-1}<v_{i}\}.
\end{equation}
This set coincides with the set of admissible labels for
$\mathbf{v}$ with respect to a Sperner's labeling of the simplex
$R_{k,q}^\pi$.

\section{Graph $G_\pi$, hypersimplices and a dihedral action}
\subsection{Motivation, definition and characterization of $G_\pi$}
For a permutation $\pi\in \mathbb{S}_{k-1}$ and $q=|S(\pi)|+2$,
the graph $G^\pi_{k,|S(\pi)|+2}$ consists of $k$ vertices,
corresponding to the vertices of $R^\pi_{k,|S(\pi)|+2}$. If we
denote the vertex $\mathbf{w}_\pi^{(i)}$ by $k+1-i\in [k]$, then
the vertex set of this graph is $[k]$. Two vertices labeled by
$i,j\in [k]$, with $i<j$, are adjacent in $G^\pi_{k,|S(\pi)|+2}$
if and only if the set $\{i,i+1,\ldots,j-1\}$ appears as sequence
of consecutive entries in $\pi$. From equation (\ref{listaboja}),
it follows that the vertex labeled by $i$ corresponds to the
hyperedge of $H^\pi_{k,|S(\pi)|+2}$, which can be monochromatic if
and only if all of its vertices are colored by $i$.

\begin{defn}
For a permutation $\pi=\pi_1\pi_2\ldots\pi_{k-1}\in
\mathbb{S}_{k-1}$, we define the graph $G_{\pi}$ on the vertex set
$[k]$. Two vertices $i<j$ are adjacent in $G_{\pi}$ if and only if
the set $\{i,i+1,\ldots,j-1\}$ appears as a contiguous subsequence
in $\pi$; that is, there exists an index $a$ such that
$\{\pi_a,\pi_{a+1},\ldots,\pi_{a+j-i-1}\}=\{i,i+1,\ldots,j-1\}$.
\end{defn}
 The graphs $G_\pi$ defined above also
appear
 in \cite{Bag} as geometric representations
 of posets associated with permutations.
  More details about these posets can be
find in \cite{Tenner}. We study these graphs here, because each
$G_{\pi}$ serves as a building block of $G^{\pi}_{k,q}$. Moreover,
understanding the structure and the properties of these graphs aid
in solving the Hypergraph Labeling Problem for hypergraphs defined
by certain permutations.

\begin{rem}
The graph $G^{\mathrm{Id}}_{k,q}$, the line graph of the
Simplex-Lattice Hypergraph $H_{k,q} = H^{\mathrm{Id}}_{k,q}$, is
also the 1-skeleton of the simplicial complex $
H^{\mathrm{Id}}_{k,q-1}$. Alternatively, it can be assembled from
edge-disjoint copies of the complete graph$K_k \cong
G_{\mathrm{Id}}$, one for each $\mathbf{w} \in W_{k,q-2}$, i.e.,
$G^{\mathrm{Id}}_{k,q} \cong \bigcup_{w \in W_{k,q-2}} (w +
G_{\mathrm{Id}})$.

For any $\pi \in \mathbb{S}_{k-1}$, the graph $G^\pi_{k,q}$
similarly contains all copies $ w + G_\pi$, for $w \in W_{k,q-2}
$, covering all vertices: $ G^\pi_{k,q} \supset \bigcup_{w \in
W_{k,q-2}} (w + G_\pi)$. \\If $ \pi \notin \{ \mathrm{Id},
\mathrm{Id}^{\mathrm{op}} \} $, extra edges of the form $ E_{a,b}$
(see~(\ref{E:edgesss})) may also appear.
\end{rem}

We collect some basic facts about graphs $G_\pi$ that follow
immediately from the definition.
\begin{prop}\label{P:propertiesofG}
Let $\pi=\pi_1\pi_2\ldots\pi_{k-1}\in \mathbb{S}_{k-1}$. Then:
\begin{enumerate}
    \item[(i)] $G_{\mathrm{Id}}=K_k$. If we let
    $\pi^{\mathrm{op}}=\pi_{k-1}\pi_{k-2}\ldots\pi_{1}$,
    then $G_\pi=G_{\pi^{\mathrm{op}}}$.
    \item[(ii)]
    $G_\pi$ always contains a $k$-cycle
    as a subgraph. Moreover,
    $G_{\pi}=C_k$ if and only if $\pi$ is a simple
    permutation.  \item[(iii)] The number of edges in $G_\pi$ is
    $cs_1(\pi)$.
    \item[(iv)]\label{F:nonintersecting} If $1\leq a<b<c<d\leq k$
    and both $ac$ and $bd$ are edges of $G_\pi$,
    then the induced subgraph on $\{a,b,c,d\}$ is a $K_4$.
\end{enumerate}
\end{prop}
If we identify the vertices of $G_\pi$ with those of a convex
$k$-gon $P_k = A_1A_2\ldots A_k$, where each $i \in [k]$
corresponds to $A_i$, then the edges of $G_\pi$ fall into three
categories:
\begin{itemize}
    \item \textbf{Polygon edges.} All edges of
    $P_k$ (i.e., the boundary edges) are always present in $G_\pi$,
    for any $\pi \in \mathbb{S}_{k-1}$.
    \item \textbf{Non-crossing diagonals} of $P_k$.
     Certain diagonals of $P_k$
    that do not intersect
    in the interior of $P_k$ may also be edges of $G_\pi$.
    \item  \textbf{Crossing diagonals} of $P_k$.
    For some permutations $\pi\in \mathbb{S}_{k-1}$, the graph
     $G_\pi$ may include edges corresponding to diagonals of $P_k$ that
     intersect in
    the interior of $P_k$.
\end{itemize}
\begin{lem}The non-crossing diagonals
    appear in all graphs $G_\pi$, except when $G_\pi=K_k$ or $G_\pi= C_k$.
\end{lem}
\begin{proof}Let $K_p = A_{i_1}A_{i_2}\ldots A_{i_p}$ be
a maximal clique in $G_\pi$ that includes at
least one diagonal of the polygon $P_k$.
 Without loss of generality, assume that the edge
 $A_{i_1}A_{i_2}$ is a such diagonal. We claim that
 $A_{i_1}A_{i_2}$ is a non-crossing diagonal.
Suppose, for contradiction, that it is crossed by another diagonal
$XY$ of $P_k$, and that $X$ lies between $A_{i_1}$ and $A_{i_2}$
in the cyclic order of the polygon. The edge $XY$ crosses the edge
$A_{i_1}M$ for every vertex $M$ of $K_p$ between $A_{i_2}$ and $Y$
and the edge $A_{i_2}N$ for every vertex $N$ of $K_p$ between $Y$
and $A_{i_1}$. By property (iv) of Proposition
\ref{P:propertiesofG}, it follows that $XA_{i_s} \in E(G_\pi)$ for
every $s = 1, 2, \ldots, p$, so $G_\pi$ has a $(p+1)$-clique. This
contradicts the maximality of $K_p$, hence, the diagonal
$A_{i_1}A_{i_2}$ must be non-crossing.
\end{proof}

 \noindent Using the fact that $\mathbb{S}_3$ contains
 no simple permutations, along with the previous
 lemma and property (iv) of Proposition \ref{P:propertiesofG},
 we obtain a characterization of the graphs in the family
 $\mathcal{G}_k = \{G_{\pi} : \pi \in \mathbb{S}_{k-1}\}$.
  A similar result appears in \cite[Theorem 3.10]{Bag}.

\begin{thm}For any $\pi \in \mathbb{S}_{k-1}$,
the graph $G_\pi$ defines a dissection of a convex $k$-gon into
smaller polygons using non-crossing
 diagonals. Each resulting $p$-gon in the dissection
 corresponds to a subgraph of $G_\pi$ that is either
 a $p$-cycle (for $p \ne 4$) or a complete graph on
 its vertices. Moreover, every
graph formed in this way arises as $G_\pi$
for some $\pi \in \mathbb{S}_{k-1}$.
\end{thm}
\begin{proof} We will use the induction on $k$.
The cases $k=3,4,5$ can be verified directly. Fix $\pi\in
\mathbb{S}_{k-1}$ and identify the vertices of $G_\pi$ with with
those of a convex $k$-gon $P_k=A_1A_2\ldots A_k$. First draw all
edges of $P_k$, followed by all edges that correspond to
non-crossing diagonals. These diagonals dissect $P_k$ into smaller
convex polygons. Consider a $p$-gon $A_{i_1}A_{i_2}\ldots A_{i_p}$
in this dissection, with $1\leq i_1<i_2<\cdots<i_p\leq k$. Its
boundary edges
$A_{i_1}A_{i_2},A_{i_2}A_{i_3},\ldots,A_{i_{p-1}}A_{i_p} $
 (either edges of $P_k$ or non-crossing
diagonals) correspond to substrings of $\pi$ that contain
consecutive sets of integers:
$$X_1=\{i_1,\ldots,i_2-1\},X_2=\{i_2,\ldots,i_3-1\},
\ldots,X_{p-1}=\{i_{p-1},\ldots,i_p-1\}.$$ Since $A_{i_1}A_{i_p}$
is also an edge of $G_\pi$, the union $X = \{i_1, \ldots, i_p -
1\}$ appears as a contiguous subsequence in $\pi$. Each $X_j$ thus
appears as a set of consecutive entries within this substring.
Replacing each $X_j$ by its index $j$ defines a new permutation
$\sigma \in \mathbb{S}_{p-1}$ that encodes the order of these
sets. Then $G_\sigma$ is a subgraph of $G_\pi$, and by the
inductive hypothesis, we have $G_\sigma \cong K_p$ or $G_\sigma
\cong C_p$. If $p = 4$, then all diagonals of the quadrilateral
are present in $G_\pi$, because $\mathbb{S}_3$ contains no simple
permutations.

To prove the converse, suppose $G$ is a graph obtained by
subdividing $P_k$ and then turning each region into a clique or
leaving it as a cycle. We show that $G = G_\pi$ for some $\pi \in
\mathbb{S}_{k-1}$, again by induction on $k$. Note that $G = K_k$
is realized by $\pi = \mathrm{Id}$, and $G = C_k$ corresponds to
any simple permutation $\pi$ (when $k \neq 4$). Suppose $G$
contains at least one non-crossing diagonal $ij$ with $1 \leq i <
j \leq k$. Then $G$ can be decomposed into two graphs $G_1$ and
$G_2$ joined along the edge $ij$. By the inductive hypothesis,
$G_1 \cong G_\alpha$ and $G_2 \cong G_\beta$ for some permutations
$\alpha \in \mathbb{S}_{k-j+i}$ and $\beta \in \mathbb{S}_{j-i}$.

To construct $\pi \in \mathbb{S}_{k-1}$ such that $G_\pi = G$,
proceed as follows:
\begin{itemize}
    \item[(1)] Keep entries of $\alpha$ less than $i$ unchanged.
    \item[(2)] Increase each entry of $\alpha$ greater than $i$ by $j - i - 1$.
    \item[(3)] Increase all entries of $\beta$ by $i - 1$, and replace
    the entry $i$ in $\alpha$ by the entire permutation
    $\beta$ or $\beta^{\mathrm{op}}$.
\end{itemize}

\begin{rem}\label{R:kada mozes lijepiti samo jednu}
In step (3), if both regions adjacent to $ij$ in $G_1$ and $G_2$
are cliques, then only one of $\beta$ or $\beta^{\mathrm{op}}$
results in a valid $\pi$ with $G_\pi = G$. The other introduces an
edge of $G_\pi$ crossing $ij$.
\end{rem}
This completes the proof.
\end{proof}

 Let $G\in \mathcal{G}_k$ be a graph that subdivides
 the convex polygon $P_k$
   into smaller regions
   using non-crossing diagonals. Each region is either a
   \emph{complete}
    (with all internal diagonals present) or a \emph{cycle} (with none).
    Two complete regions $R$ and $R'$ \emph{connected} if there
exists a sequence $R=R_1,R_2,\ldots,R_p=R'$ of complete regions,
each sharing an edge with the next.

The number of graphs in $\mathcal{G}_k$ is the number of all
interval posets, see Theorem 22 in \cite{Bouvel}. For a fixed
graph $G \in \mathcal{G}_k$, we want to determine the number of
permutations $\pi \in \mathbb{S}_{k-1}$ such that $G_\pi \cong G$.
Recall that for $G=C_k$, the number of such permutation is $s_k$,
the number of simple permutations in $\mathbb{S}_{k-1}$, see
\cite{Bri}. For $1\leq k\leq 9$, we have that $s_k=1, 2, 0, 2, 6,
46, 338, 2926, 28146$ (OEIS A111111~\cite{Sloan}).

\begin{thm}\label{T:BBB} Let $G\in \mathcal{G}_k$
be a graph that contains $c_i$ polygons that are cycles of length
$i>4$, and $t$ connected components consisting of entirely
complete regions (graphs). Then:
$$\Big|\{\pi\in
    \mathbb{S}_{k-1}:G_\pi=G\}\Big|=2^t s_5^{c_5}s_6^{c_6}
    \cdots s_m^{c_m}.$$
\end{thm}
\begin{proof}We proceed by induction on the number of regions in $G$.
If $G$ consists only of cycles (i.e., no complete regions), we
construct $G$ by successively gluing cycles along non-crossing
diagonals. In this case, the count follows directly from
Remark~\ref{R:kada mozes lijepiti samo jednu}, since each cycle of
length $i>4$ contributes $s_i$.

 If $G$ contains complete graph regions, we first glue together all
 such regions within each connected component.
 By the inductive hypothesis, each complete region contributes exactly two permutations. Remark~\ref{R:kada mozes lijepiti samo jednu}
 then implies that the entire component admits exactly two permutations.
Multiplying over all $t$ connected components of complete graphs
yields the factor $2^t$.

Combining these with the contributions from the cycle regions
completes the proof.
\end{proof}
The Schr\" oder-Hipparchus number (sequence A001003 in
\cite{Sloan}) count the number of distinct ways to subdivide a
convex $k$-gon into smaller polygons using non-crossing diagonals.
These numbers also count the number of faces of all dimensions of
a $(k-3)$-dimensional associahedron.
\begin{thm}For any $k\leq 5$ and for each $G\in\mathcal{G}_k$,
there exist
exactly two permutations $\pi\in \mathbb{S}_{k-1}$ such that
$G=G_\pi$. For any $k > 5$, the number of graphs in
$\mathcal{G}_k$ that correspond to exactly two permutations is
equal to the $k$-th Schr\" oder-Hipparchus number.
\end{thm}
\begin{proof}
The claim holds for $k=3,4,5$ by direct verification. For $k>5$,
it follows from Theorem \ref{T:BBB} that a graph $G\in
\mathcal{G}_k$ is represented by exactly two permutations if and
only if all of its regions are complete graphs (i.e., there are no
cycles). The number of such graphs is equal the number of ways to
subdivide a convex $k$-gon using non-crossing diagonals. The
regions in such a subdivision correspond to complete graphs.
\end{proof}
\subsection{Graph $G_\pi$ and hypersimplex}
Recall that the simplices $\sigma_\pi$, where $\pi\in\mathbb
S_{k-1}$ has exactly $m-1$ adjacent inversions, form a
triangulation of the hypersimplex $\Delta_{k,m}$ (see Theorem
\ref{T:triangulationofHip}). The graphs $G_\pi$ admit an
additional elegant interpretation in the context of this
triangulation.

\begin{prop}
Let $\pi \in \mathbb{S}_{k-1}$ be a permutation with exactly $m-1$
descents. Then the graph $G_\pi$ is equal to the intersection of
the simplex $\sigma_{\pi^{-1}}$ with the $1$-skeleta of the
hypersimplex $\Delta_{k,m}$, i.e., $G_\pi=\sigma_{\pi^{-1}}\cap
\Delta^{(\leq 1)}_{k,m}.$
\end{prop}

\begin{proof}
 By Proposition \ref{P:adjinv=desc}, the inverse permutation
 $\pi^{-1}$ has exactly $m-1$ adjacent inversions. Let
$S_{\pi^{-1}}=(S_1,S_2,\ldots,S_k)$ be the sorted family
associated with $\pi^{-1}$, as defined in (\ref{sortirana
familija}). By Theorem \ref{T:triangulationofHip}, it follows that
$\sigma_{\pi^{-1}}=\operatorname{conv}\{e_{S_1}
,e_{S_2},\ldots,e_{S_k}\}\subset \Delta_{k,m}$. We claim that $ij$
is an edge of $G_\pi$ if and only if the segment
$[e_{S_i},e_{S_j}]$ is an edge of $\Delta_{k,m}$. Recall that
$[e_{S_i},e_{S_j}]$ is an edge of $\Delta_{k,m}$ if and only if
$e_{S_i}-e_{S_j}=e_s-e_t$ for some $s,t\in[k],s\neq t$.

Let $i,j\in[k],i<j$. We have that
\[e_{S_i}-e_{S_j}=\sum_{l=i}^{j-1}
(e_{S_l}-e_{S_{l+1}})=\sum_{l=i}^{j-1}(e_{\pi^{-1}_{l}}-
e_{\pi^{-1}_{l}+1}).\] The last expression is equal to
$e_s-e_{s+j-i}$
 if and only if
$$\{\pi^{-1}_i,\pi^{-1}_{i+1},\ldots,\pi^{-1}_{j-1}\}=
\{s,s+1,\ldots,s+j-i-1\}.$$
 This condition holds if and only if
  $\{i,i+1,\ldots,j-1\}$ form a contiguous block in $\pi$,
   which happens if and only if $ij$ is an edge in the graph $G_\pi$.
\end{proof}
The next corollary characterize graphs $G_\pi$ for all
permutations with exactly one descent.
\begin{cor}If $\pi\in \mathbb{S}_{k-1}$
has only one descent, then $G_\pi$ can be obtained by gluing a few
complete graphs on the edges of a polygon with an odd number of
vertices.
\end{cor}

For a permutation $\pi=\pi_1\pi_2\ldots\pi_{k-1}\in
\mathbb{S}_{k-1}$, define
$$\verb"ear"(\pi)=\big|\{i:|\pi_{i+1}-\pi_i|=1\}
\big|+\big|\{1,k-1\}\cap\{\pi_1,\pi_{k-1}\}\big|.$$ Let us refer
to the diagonals $\{A_iA_{i+2} : i \in [k]\}$ (with indices taken
modulo $k$) of the polygon $P_k = A_1A_2\ldots A_k$, in which the
graph $G_\pi$ is embedded, as \emph{short diagonals}. Then
$\mathrm{ear}(\pi)$ counts exactly the number of short diagonals
that appear as edges in the graph $G_\pi$.

 If $|S(\pi)|=m-1$, then $\sigma_\pi$ appears in the
triangulation of $\Delta_{k,m}$. We say that $\sigma_\pi$ is
\emph{adjacent }to a facet of $\Delta_{k,m}$ if their intersection
is a facet of $\sigma_\pi$.

\begin{thm} Let
$\pi=\pi_1\pi_2\ldots\pi_{k-1}\in \mathbb{S}_{k-1}$ be a
permutation such that $|S(\pi)|=m-1$. Then the simplex
$\sigma_\pi$ is adjacent to exactly $\verb"ear"(\pi)$ facets of
the hypersimplex $\Delta_{k,m}$.
\end{thm}
\begin{proof}
The vertices of $\sigma_\pi$ are given by (\ref{sortirana
familija}) and described further in
Remark~\ref{R:rasporedusorted}. Using this, together with the
description of the facets of $\Delta_{k,m}$ in (\ref{E:facets}),
we list the conditions under which $\sigma_\pi$ is adjecent to a
facet of $\Delta_{k,m}$:
\begin{itemize}
    \item $\sigma_{\pi}$ is adjacent to $F_1$ if and only if
    $\pi_1=1$.
    \item $\sigma_{\pi}$ is adjacent to $F_i$ if and only if
    $i-1$ is immediately before $i$ in $\pi$.
    \item $\sigma_{\pi}$ is adjacent to $F_k$ if and only if
    $\pi_{k-1}=k-1$.
\end{itemize}
Similarly, we obtain that
\begin{itemize}
    \item $\sigma_{\pi}$ is adjacent to $H_1$ if and only if
    $\pi_{k-1}=1$.
    \item $\sigma_{\pi}$ is adjacent to $H_i$ if and only if
    $i-1$ is immediately after $i$ in $\pi$.
    \item $\sigma_{\pi}$ is adjacent to $H_k$ if and only if
    $\pi_{1}=k-1$.
\end{itemize}
In total, these correspond precisely to the definition of
$\operatorname{ear}(\pi)$, completing the proof.
\end{proof}
The number of simplices in the triangulation of $\Delta_{k,m}$
that are adjacent to the facet $F_i$ is $A(k-1,m-1)$, while the
number adjacent to the facet $H_i$ is $A(k-1,m-2)$.

A maximal simplex $\sigma_\pi$ in the triangulation of
$\Delta_{k,m}$ is called an \emph{internal simplex} if it is not
adjacent to any facet of $\Delta_{k,m}$.
\begin{cor} Let
$\pi=\pi_1\pi_2\ldots\pi_{k-1}\in \mathbb{S}_{k-1}$ be a
permutation such that $|S(\pi)|=m-1$. Then the simplex
$\sigma_\pi$ is an internal simplex in the triangulation of
$\Delta_{k,m}$ if and only if the following conditions are
satisfied:
\begin{enumerate}
    \item For all $i\in [k-1]$, the numbers
    $i$ and $i+1$ are not adjacent in $\pi$.
    \item $\{\pi_1,\pi_{k-1}\}\cap \{1,k-1\}=\emptyset.$
\end{enumerate}

\end{cor}

For $ k = 2m + 1$, the unique permutation in $ \mathbb{S}_{k-1}$
with exactly one adjacent inversion satisfying the given
conditions is
$$
\pi = (m+1)\,1\, (m+2)\, 2\, (m+3)\, \cdots\, (m-1)\, (2m)\, m,
$$
as shown in \cite{Lam}. The total number of interior simplices in
$\Delta_{k,m}$, summed over all $m$, equals the number of circular
arrangements of $ 1,2,\ldots,k $ such that adjacent numbers never
differ by 1 modulo $k$; see OEIS A002816~\cite{Sloan}.

Hertzsprung's problem counts the number of ways to place $k$
non-attacking kings on a $k \times k$ chessboard, with one per row
and column, or equivalently, the number of permutations of length
$k $ with no rising or falling successions (OEIS
A002464~\cite{Sloan}). A variant of Hertzsprung's problem asks for
the number of such arrangements on a $ (k{-}1)\times(k{-}1)$ board
with the four corners removed, again with one king per row and
column. This also equals the total number of interior simplices in
$\Delta_{k,m} $ for $ 1 \leq m \leq k{-}1$.

\subsection{Dihedral group action}

We identify the vertex set $[k]$ of $G_{\pi}$ with the vertices of
a regular $k$-gon, also denoted by $[k]$. The dihedral group
$D_k=\{e,r,r^2,\ldots,r^{k-1},s,rs,\ldots,r^{k-1}s\}$ acts
naturally on $[k]$ by $r(i)=i+1$ and $s(i)=k+2-i$, for all $i\in
[k]$, with arithmetic taken modulo $k$. This action of $D_k$ on
$[k]$ induces a corresponding action on the set
$\mathcal{G}_k=\{G_\pi:\pi \in \mathbb{S}_{k-1}\}$ by rotating and
reflecting the graphs $G_\pi$. For an element $d\in D_k$ and a
subset $X=\{a_1,a_2,\ldots,a_m\}\subset [k]$, we define
$d(X)=\{d(a_1),d(a_2),\ldots,d(a_m)\}$, which gives an action of
$D_k$ on the collection of $m$-element subsets of $[k]$.

Our goal is to define an action of $D_k$ on $\mathbb S_{k-1}$ that
is compatible with the induced action on $\mathcal{G}_k$, that is,
we seek an action such that $G_{d(\pi)}=d(G_{\pi})$ for all $d\in
D_k$ and $\pi\in\mathbb S_{k-1}$. It suffices to define this
action on the generators $r$ and $s$.

\begin{defn}\label{D:actiononS} Let $\pi=\pi_1\ldots\pi_{t-1}\,
(k-1)\, \pi_{t+1}\ldots\pi_{k-1}\in\mathbb S_{k-1}$. We define the
\emph{rotation} of $\pi$ (i.e., the action of $r$ on $\pi$) by
\[r(\pi)=(\pi_{t+1}+1)
\ldots(\pi_{k-1}+1)\,1\,(\pi_1+1)\ldots(\pi_{t-1}+1).\] For
$\pi=\pi_1\ldots\pi_{i-1}\, 1\,\pi_{i+1}\ldots\pi_{k-1}\in\mathbb
S_{k-1}$, we define the \emph{symmetry} (or reflection) of $\pi$
(i.e., the action of $s$ on $\pi$) by
\[s(\pi)=(k+1-\pi_{i-1})\ldots(k+1-\pi_{1})1(k+1-\pi_{k-1})
\ldots(k+1-\pi_{i+1}).\]
\end{defn}
\begin{prop}\label{P:DejstvoD-knaS} The maps $r$ and $s$
defined in Definition~\ref{D:actiononS}
 generate an action of $D_k$ on $\mathbb S_{k-1}$.
\end{prop}

\begin{proof} It suffices to verify that
the defining relations of $D_k$ are satisfied, namely,
\begin{equation}\label{E:identiteta}
r^k=s^2=(rs)^2=\mathrm{Id}_{\mathbb S_{k-1}},
\end{equation} where $\mathrm{Id}_{\mathbb S_{k-1}}$
denotes the identity map. We appeal to the following geometric
interpretation of the maps $r$ and $s$. Given a permutation
$\pi=\pi_1\pi_2\ldots\pi_{k-1}\in \mathbb{S}_{k-1}$, place the
elements $\pi_1,\pi_2,\ldots,\pi_{k-1},k$ at the vertices of a
regular $k$-gon in the clockwise order. To compute $r(\pi)$,
 increment each label $i\in[k]$ by $i+1$ modulo $k$, and then read the
resulting sequence starting at the position labeled $k$,
proceeding clockwise. To obtain $s(\pi)$, replace each label
$i\in[k]$ with $k+1-i$ and then read the sequence starting at $k$,
proceeding counterclockwise. Under this interpretation, the
identites in (\ref{E:identiteta}) follow by a direct calculation.
\end{proof}
\begin{thm} The map $\pi\mapsto G_\pi$ is $D_k$-equivariant,
that is for all $d\in D_k$ and $\pi\in \mathbb{S}_{k-1}$, we have
$G_{d(\pi)}=d(G_\pi)$.
\end{thm}
\begin{proof} It suffices to verify the identities
$r(G_\pi)=G_{r(\pi)}\textrm{ and }s(G_\pi)=G_{s(\pi)}$, for every
$ \pi\in\mathbb S_{k-1}.$ Let $1<i<j\leq k$. Then we have have the
following chain of equivalences:
$$ij\in E(r(G_\pi)) \Leftrightarrow (i-1)(j-1)\in E(G_\pi)
\Leftrightarrow\{i-1,\ldots,j-2\} \textrm{ are consecutive in }
\pi.$$ From the definition of $r(\pi)$, it follows that
$$\{i-1,\ldots,j-2\} \textrm{ are consecutive in }
\pi\Leftrightarrow \{i,\ldots,j-1\} \textrm{ are consecutive in }
r(\pi),$$ which is equivalent with $ij\in G_{r(\pi)}$. Next,
consider $1j\in E(r(G_\pi)$. Then:
$$1j\in E(r(G_\pi)) \Leftrightarrow (j-1)k\in E(G_\pi)
\Leftrightarrow\{j-1,\ldots,k-1\} \textrm{ are consecutive in }
\pi.$$ Using the geometric interpretation of $r$ from the proof of
Proposition~\ref{P:DejstvoD-knaS}, this is equivalent to
$$\{j-1,\ldots,k-1\} \textrm{ are consecutive in }
\pi\Leftrightarrow \{1,2,\ldots,j-1\} \textrm{ are consecutive in
} r(\pi).$$ Therefore, $1j\in r(G_\pi)\Leftrightarrow 1j\in
G_{r(\pi)}$. The identity $s(G_\pi)=G_{s(\pi)}$ follows by a
similar argument, and we omit the proof.
\end{proof}
We now describe the permutations in $\mathbb{S}_{k-1}$ that are
invariant under the action of the generators $r$ and $s$ of the
dihedral group $D_k$. Suppose that $\pi\in \mathbb{S}_{k-1}$
satisfies $\pi_a=k-1$ for some $a\in [k-1]$. From the geometric
interpretation of the action of $r$ on $\mathbb{S}_{k-1}$ (see
Proposition~\ref{P:DejstvoD-knaS}) it follows that if
$r(\pi)=\pi$, then necessarily $\pi_{2a}=k-2$. Continuing in this
fashion, we find that $r(\pi)=\pi$ if and only if the set
$\{a,2a,\ldots,(k-1)a,ka\}$ modulo $k$ forms a complete residue
system modulo $k$. This condition is equivalent to $gcd(a,k)=1$,
and the number of $r$-invariant permutation in $\mathbb{S}_{k-1}$
is given by Euler's totient function: $|\{\pi\in
\mathbb{S}_{k-1}:r(\pi)=\pi\}|=\varphi(k)$. The next Remark
describes all $r$-invariant permutations in $\mathbb{S}_{k-1}$.

\begin{rem}\label{R:Blocks}
\label{P:blocksinr-invariant} Let $m<k$ be an integer relatively
prime to $k$. Assume that $k=mt+s$ for some integers $t \geq 0$
and $0<s<m$. Define a permutation $\pi \in \mathbb{S}_{k-1}$ as
follows: set $\pi_m=1$, and $\pi_{im}=i$ for all $i = 2, \ldots, k
- 1$, where the indices $im$ are taken modulo $k$. Note that $\pi$
is an $r$-invariant permutation, and in particular we have
$\pi_{k-m}=k-1$, and $\pi_s=k-t$.

Furthermore, $\pi$ has exactly $m$ almost equicardinal blocks.
Namely, $s-1$ of these blocks (those starting with
$\pi_1,\ldots,\pi_{s-1}$) are of length $t+1$, and the remain
$m-s+1$ of them is of length $t$. Since the block $B_1$ (starting
with $1$) has $t$ elements, we conclude that $\pi_{m-s}=t+1$
\end{rem}

 Clearly, all
$r$-invariant permutations have the same graph $G_\pi\cong C_k$.
However, the converse is not true. For instance, in
$\mathbb{S}_6$, there are $46$ simple permutations whose
associated graph is $C_7$, but only $4$ of them ($362514$,
$246135$, $ 531642$ and $415263$) are fixed by $r$.

Note that the $r$-invariant permutations in $\mathbb{S}_{k-1}$ are
precisely those whose vectors of consecutiveness attain the
extremal bounds described in (\ref{E:boundsforcs}).

\begin{thm}\label{T:rinv}Let $m\leq\frac{k-1}{2}$
 be an integer such that $gcd(m,k)=1$, and
let $\pi\in \mathbb{S}_{k-1}$ be an $r$-invariant permutation
satisfying $\pi_m=1$. Then the vector of consecutiveness of $\pi$
is given by \begin{equation}\label{E:extremal} cs_i(\pi)=k\textrm{
for all }i\in[m-1],\textrm{ and } cs_m(\pi)= {k\choose 2}-(m-1)k
\end{equation}.
\end{thm}
\begin{proof} An $r$-invariant permutation $\pi\in
\mathbb{S}_{k-1}$ is uniquely determined by the position of $1$,
see Remark \ref{R:Blocks}. The permutations $\pi$ and
$\pi^{\mathrm{op}}$ have the same $cs$-vector, so we may assume
that $m\leq\frac{k-1}{2}$.

For each $i<m$, every set of $i$ consecutive entries in $\pi$
(there are exactly $k-i$ such sets) decomposes into a disjoint
union of $i$ intervals (singletons) of consecutive integers.
Similarly, for each such $i$, there are exactly $i$ sets of length
$k-i$, and each one decomposes into $i$ intervals of consecutive
integers.

 Now, consider any set of $j$
consecutive entries in $\pi$ with $k-m\geq j\geq m$. There are
exactly $k-j$ of such sets, and each one
 decomposes into a disjoint union of exactly $m$ of consecutive
integers (entries from the same block of $\pi$ are adjacent).
\end{proof}
Assume that the vector of consecutiveness of a permutation $\pi\in
\mathbb{S}_{k-1}$ satisfies $(\ref{E:boundsforcs})$ with
equalities. Then, by the argument presented in the proof of the
previous theorem, it also satisfies (\ref{E:extremal}) for some
$m$. We conclude that, in this case, $\pi$ is $r$-invariant and
$\pi_m=1$.

To characterize $s$-invariant permutations,
 suppose $\pi_t=1$. From the definition of $s$,
we deduce that $s(\pi)=\pi$ if and only if
$$\pi_i=k+1-\pi_{t-i}\textrm{ for all }i\in [t-1 ]\textrm{, and }$$
$$\pi_j=k+1-\pi_{k+t-j}\textrm{ for all }j=t+1,\ldots,k-1.$$
In the case when $k=2m$, the position of elements
$\{2,3,\ldots,m-1\}$ in $\pi$ determine the positions for
$\{m,m+1,\ldots,2m-1\}$. If $k$ is even, then $s(\pi)=\pi$ implies
that $t=\pi^{-1}(1)$ is odd. A simple counting of all admissible
arrangement gives us that the number of $s$-invariant permutations
in $\mathbb{S}_{2m-1}$ is $m!2^{m-1}$. A similar reasoning for odd
$k=2m-1$ shows that the number of $s$-invariant permutation in
$\mathbb{S}_{2m-2}$ is $(m-1)!2^{m-1}$.

 Now we describe how the action of $D_k$ on $[k]$
(and on its $m$-element subsets) extends to an action on the set
of sorted families of subsets of $[k]$. Given a sorted family
$(S_1,S_2,\ldots,S_t)$, an element $d\in D_k$ acts by mapping each
set $S_i$ to $d(S_i)$, producing the collection
$\{d(S_1),\ldots,d(S_t)\}$. The following lemma determines the
order in which these transformed sets appear in the resulting
sorted family.
\begin{lem} Let $(A,B)$ be a
sorted pair of subsets of $[k]$. Then
\begin{itemize}
    \item[(i)] If $k\in B\setminus A$, then $(r(B),r(A))$ is sorted;
    otherwise $(r(A),r(B))$ is sorted.
    \item[(ii)] If $1\in A\setminus B$, then $(s(A),s(B))$ is
sorted; otherwise, $(s(B),s(A))$ is sorted.
\end{itemize}
\end{lem}
\begin{proof}Assume that
$A=\{a_1,a_2,\ldots,a_m\}$ and $B=\{b_1,b_2,\ldots,b_m\}$, where
$1\leq a_1\leq b_1\leq a_2\leq b_2\ldots\leq a_m\leq b_m\leq k$.

\noindent $(i)$ Suppose $k\in B\setminus A$. Then $b_m=k>a_m\geq
b_{m-1}$, and since $r(i)=i+1$ modulo $k$, we have $$r(b_m)\leq
r(a_1)\leq r(b_1)\leq \cdots \leq r(b_{m-1})\leq r(a_m).$$ If
instead $a_m=b_m=k$, or $b_m<k$, a similar reasoning shows that
$(r(A),r(B))$ is sorted. This completes the proof of $(i)$.

\noindent $(ii)$ Suppose that $1\in A\setminus B$, so $a_1=1<b_1$
and $s(a_1)=1<s(b_m)$. More generally, since $s(i)=k+2-i$ for
$i\neq 1$, the inequalities $b_i\leq a_{i+1}\Rightarrow
s(a_{i+1})\leq s(b_i)$ imply
$$s(a_1)\leq s(b_m)\leq s(a_{m-1})\leq s(b_{m-1})\leq
\cdots\leq s(a_2)\leq s(b_1),$$ and the pair $(s(A),s(B))$ is
sorted. It is easy to check that if $1\notin A\setminus B$ then
$(s(B),s(A)$ is sorted.
\end{proof}
As a direct consequence of the previous lemma, we obtain the
following description of the $D_k$-action on sorted families of
subsets of $[k]$.
\begin{prop}\label{dejstvonasort} Let $\mathcal S=(S_1,S_2,\ldots,S_k)$
be a sorted family of subsets of $[k]$.
\begin{itemize}
    \item[(a)] If $k\in (S_{i+1}\cap
    \cdots\cap S_k)\setminus (S_1\cup\cdots\cup
S_i)$ then \[r(\mathcal S)=(r(S_{i+1}),r(S_{i+2}),\ldots,
r(S_k),r(S_1),r(S_2),\ldots,r(S_i))\]
    \item[(b)] If $1\in (S_1\cap\cdots\cap S_j)\setminus
(S_{j+1}\cup\cdots\cup S_k)$ then \[s(\mathcal
S)=(s(S_{j}),s(S_{j-1}),\ldots,
s(S_1),s(S_{k}),s(S_{k-1}),\ldots,s(S_{j+1}))\]
\end{itemize}
\end{prop}

 Let $\mathcal S_\pi$ denote
 the maximal sorted family associated with $\pi\in\mathbb S_{k-1}$,
 as defined in (\ref{sortirana familija}).

\begin{cor}\label{ekvivarijantno}
The map $\pi\mapsto\mathcal S_\pi$ is $D_k$-equivariant.
\end{cor}

\begin{proof} Assume that $\mathcal
S_\pi=(S_1,S_2,\ldots,S_k)$ is a sorted family defined by the
permutation $\pi=\pi_1\ldots\pi_{t-1} (k-1) \pi_{t+1}
\ldots\pi_{k-1}\in\mathbb S_{k-1}$.

Then
 $k\in (S_{t+1}\cap \ldots\cap S_k)\setminus
  (S_1\cup\ldots\cup S_t)$. By Proposition~\ref{dejstvonasort},
  the image under $r$ is the sorted family
\[r(\mathcal S_\pi)=(r(S_{t+1}),r(S_{t+2}),
\ldots,r(S_k),r(S_1),r(S_2),\ldots,r(S_t)).\] For each
$i\in\{t,t+1,\ldots,k-2\}$, we compute
$$r(S_{i+1})\setminus r(S_{i+2})= r(S_{i+1}\setminus
S_{i+2})=r(\pi_{i+1})= \pi_{i+1}+1.$$ Similarly,
 $r(S_k)\setminus r(S_1)=r(S_k\setminus S_1)=r(k)=1$, and for
 all $i\in [t-1]$, $r(S_i\setminus S_{i+1})=r(S_i\setminus S_{i+1})
 =r(\pi_i)=\pi_{i}+1$.
 Thus $r(\mathcal S_\pi)=\mathcal S_{r(\pi)}$.
 The proof for $s$ follows analogously.

\end{proof}

The following theorem shows that the graphs $G_{k,q}^\pi$ and
$G_{k,q}^\omega $
 are isomorphic whenever $\pi$ and $\omega$ lie in the same orbit under
 the action of $D_k$ on $\mathbb S_{k-1}$.
 Moreover, this isomorphism preserves Sperner labelings:
 it maps any Sperner labeling of
 $G_{k,q}^\pi$ to a Sperner's
 labeling of $G_{k,q}^\omega$, preserving the number of colored vertices.

\begin{thm}\label{T:izopriD_k} For every $\pi\in \mathbb{S}_{k-1}$
and every $d\in D_{k}$, there exists an isomorphism
$f:W^\pi_{k,q}\rightarrow W^{d(\pi)}_{k,q}$ such that
$L^{d(\pi)}(f(\mathbf{x}))=d(L^\pi(\mathbf{x}))$ for every
$\mathbf{x}\in W_{k,q}^\pi$.
\end{thm}

\begin{proof}

 Let $\pi\in\mathbb S_{k-1}$, and let $S(\pi)$ denote the
 set of adjacent inversions. Let $\mathcal{S}_\pi=(S_1,S_2,\ldots,S_k)$
 be the sorted family associated with $\pi$.

  For this permutation $\pi$, define the map
  $\varphi_\pi:W_{k,q}^\pi\to \mathbb Z^k$ by
\[\varphi_\pi(\mathbf{x})=(x_1,x_2-x_1,\ldots,x_{k-1}-x_{k-2},q-x_{k-1})-
e_{S_k}.\]Since
 $x_i-x_{i-1}>0$ for all $i-1\in S(\pi)$, and $x_{k-1}<q$, the map
 $\varphi_\pi$ is well-defined.
 This map is a bijection between $W_{k,q}^\pi$
  and $V_{k,q-|S(\pi)|-1}=
  \{\mathbf{a}\in\mathbb Z^k_+:a_1+\cdots+a_k=q-|S(\pi)|-1\}$
  i.e., the set of weak compositions of
  $q - |S(\pi)| - 1$ into $k$ parts.

Recall that $\mathbf{x}\mathbf{y}$ is an edge of $G_{k,q}^\pi$ if
and only if
 $\mathbf{y}=\mathbf{x}+e_{\pi_a}+e_{\pi_{a+1}}+\cdots+e_{\pi_{b-1}}
 \ (1\leq a<b\leq k)$. This is equivalent with
\begin{multline*}
\varphi_\pi(\mathbf{y})=\varphi_\pi(\mathbf{x})+(e_{\pi_a}-e_{\pi_a+1})+
(e_{\pi_{a+1}}-e_{\pi_{a+1}+1})+\cdots+(e_{\pi_{b-1}}-e_{\pi_{b-1}+1})=\\
=\varphi_\pi(\mathbf{x})+e_{S_a}-e_{S_{a+1}}+e_{S_{a+1}}-
e_{S_{a+2}}+\cdots+e_{S_{b-1}}-e_{S_b}=\varphi_\pi(\mathbf{x})+e_{S_a}-e_{S_b}.
\end{multline*}
Define the auxiliary graph $\Gamma_{k,q}^\pi$ with vertex set
$V_{k,q-|S(\pi)|-1}$, where two vertices $\mathbf{a}, \mathbf{b}$
are adjacent if there exist $i, j \in [k]$ such that $\mathbf{a} +
e_{S_i} = \mathbf{b} + e_{S_j}$. Then $\varphi_\pi$ is an
isomorphism between $G_{k,q}^\pi$ and $\Gamma_{k,q}^\pi$.

Similarly, we define $\varphi_{d(\pi)}$ and construct
$\Gamma_{k,q}^{d(\pi)}\cong G^{d(\pi)}_{k,q}$ in the same way.
  It remains to show that the graphs $\Gamma_{k,q}^\pi$ and
  $\Gamma_{k,q}^{d(\pi)}$ are isomorphic.

From Proposition \ref{dejstvonasort} and Corollary
\ref{ekvivarijantno}, we know that
 $|S(\pi)|=|S(d(\pi))|$ and that
$d(S_1),d(S_2),\ldots,d(S_k)$ are elements of the sorted family
that is associated with $d(\pi)$. Hence
$V(\Gamma_{k,q}^{d(\pi)})=V(\Gamma_{k,q}^\pi)=V_{k,q-|S(\pi)|-1}$,
i.e., both graphs have the same vertex set.

\noindent Now consider the action of $D_k$ on $\mathbb{R}^k$ via
the linear map $A_d: \mathbb{R}^k \to \mathbb{R}^k$ defined by
$A_d(e_i) = e_{d(i)}$. This map preserves the set $V_{k,q -
|S(\pi)| - 1}$ and induces an isomorphism between
$\Gamma_{k,q}^\pi$ and $\Gamma_{k,q}^{d(\pi)}$. Indeed, if
$\mathbf{a}, \mathbf{b} \in V_{k,q - |S(\pi)| - 1}$ and
$\mathbf{a} + e_{S_i} = \mathbf{b} + e_{S_j}$, then: $
A_d(\mathbf{a}) + e_{d(S_i)} = A_d(\mathbf{b}) + e_{d(S_j)}$.
Hence, the map $ f := \varphi_{d(\pi)}^{-1} \circ A_d \circ
\varphi_\pi $ is an isomorphism between $G_{k,q}^\pi$ and
$G_{k,q}^{d(\pi)}$.

To verify compatibility with Sperner labelings, note that
$\mathbf{x} \in W_{k,q}^\pi$ can be labeled with $i$ if and only
if the $i$-th coordinate of $\varphi_\pi(\mathbf{x})$ is positive,
i.e., if $(\varphi_\pi(\mathbf{x}))_i>0$. Since
$A_d(\varphi_\pi(\mathbf{x})) = \varphi_{d(\pi)}(f(\mathbf{x}))$,
we obtain that
$(\varphi_\pi(\mathbf{x}))_i=(A_d(\varphi_\pi(\mathbf{x})))_{d(i)}$,
which implies $L^{d(\pi)}(f(\mathbf{x}))=d(L^\pi(\mathbf{x}))$.

\end{proof}

\section{Distance colorings of $G^\pi_{k,q}$}
In this section, we study Sperner labelings of the graph
$G^\pi_{k,q}$ that use multiple colors. By the equivalence
established in Section 2, each Sperner labeling of $G^\pi_{k,q}$
naturally induces a corresponding Sperner labeling of the
hypergraph $H^\pi_{k,q}$. We demonstrate that, for most
permutations $\pi$, our labeling produces more monochromatic
hyperedges of $H^\pi_{k,q}$ than the greedy coloring. Moreover,
our labeling is optimal for certain permutations.
\subsection{Definition of distance coloring}
For a fixed permutation $\pi=\pi_1 \pi_2\ldots\pi_{k-1}\in
\mathbb{S}_{k-1}$, we encode its set of adjacent inversions
$S(\pi)$ by
 $\varepsilon(\pi)=(\varepsilon_1,\varepsilon_2,\ldots,\varepsilon_{k})$,
 where
 $$\varepsilon_1=\varepsilon_k=0,\textrm{ and }\varepsilon_i=\left\{%
\begin{array}{ll}
    0, & \hbox{if $i-1 \notin S(\pi)$;} \\
   1, & \hbox{if $i-1 \in S(\pi)$.} \\
\end{array}%
\right    .$$ It follows from (\ref{listaboja}) that the hyperedge
$F(\mathbf{x},\pi)$ of $H^\pi_{k,q}$ can be $i$-monochromatic if
and only if $x_{i-1}<x_i-\varepsilon_{i} \,\,(\textrm{assuming }
x_{0}=0 \textrm{ and }x_k=q-1)$.

 Assume that
$C=\{a_1,a_2,\ldots,a_{t}\} \subset [k]$ is a set of independent
vertices in the graph $G_\pi$. We can color the vertices of
$G_\pi\cong G^\pi_{k,|S(\pi)|+2}$ that belong to $C$
  with the corresponding color, while the remaining
  vertices of $G_\pi$ are colored
by $0$, as we describe in (\ref{uslov1}). Clearly, if $C$ is a
maximal independent set in $G_\pi$, then this coloring maximizes
the number of monochromatic hyperedges of $H^\pi_{k,|S(\pi)|+2}$.
 Motivated by the above example, we fix an independent set of vertices
$C=\{a_1,a_2,\ldots,a_t\}\subset [k]$ in $G_\pi$, and partition
the vertices of $W_{k,q}^\pi$ into $t+1$ groups according to their
distances from $w_\pi^{(k+1-a_i)}$. For an arbitrary $q\in
\mathbb{N}$ such that $q>|S(\pi)|+2$, we define the affine map
$D:W ^\pi_{k,q}\rightarrow \mathbb{R}^t$ by
$$D(x_1,x_2,\ldots,x_{k-1})=(x_{a_1}-x_{a_1-1}-\varepsilon_{a_1},
x_{a_2}-x_{a_2-1}-\varepsilon_{a_2},\ldots
,x_{a_t}-x_{a_t-1}-\varepsilon_{a_t} ).$$
 Note that the vertices of the simplex $R^\pi_{k,q}$ are mapped by
 $D$ as follows

$$D(\mathbf{w}^{(j)}_\pi)=\left\{%
\begin{array}{cl}
    (q-|S(\pi)|-1)e_i, & \hbox{if $k+1-j=a_i\in C$;} \\
   \mathbf{0}, & \hbox{if $k+1-j \notin C.$} \\
\end{array}%
\right.    $$
Therefore, we conclude that
$$D(R^\pi_{k,q})=L_{t,q}=(q-|S(\pi)|-1)\cdot
\operatorname{conv}\{0,e_1,e_2,\ldots,e_{t}\} .$$ The image of
$W^\pi_{k,q}$ under the map $D$ consists of integer points
contained within the $t$-dimensional simplex $L_{t,q}$.

\begin{defn}[The distance coloring]\label{D:distcol} Assume that
$C=\{a_1,a_2,\ldots,a_t\}$ is an independent set of vertices in
$G_\pi$. For each $i=1,2,\ldots,t$, we define $A_i\subset
W^\pi_{k,q}$ by
\begin{equation}\label{E:bojenje}
A_i=\big\{\mathbf{x}\in
W^\pi_{k,q}:x_{a_i}-x_{a_i-1}-\varepsilon_{a_i}>
x_{a_j}-x_{a_j-1}- \varepsilon_{a_j}\textrm{ for all } j \neq
i\big\}.
\end{equation}
In other words, a vertex $\mathbf{x}$ belongs to $A_i$ if and only
if the $i$-th coordinate of $D(\mathbf{x})$ is the unique maximum
among all coordinates. The vertices from $A_i$ are colored by
$a_i$, and the rest of vertices of $W^\pi_{k,q}$ are colored by
$0$.
\end{defn}
Note that all vertices in $A_i$ can be colored by $a_i$, and $A_i$
consists of precisely those vertices $\mathbf{x}\in W^\pi_{k,q}$
that are closer to $\mathbf{w}^{(k+1-a_i)}_\pi$ than to
$\mathbf{w}^{(k+1-a_j)}_\pi$ with $a_j\in C$, $j\neq i$.
Therefore, we refer to the coloring defined by (\ref{E:bojenje})
as the \emph{distance coloring}.

Unfortunately, the distance coloring defined above is not a
well-defined Sperner's labeling for all permutations $\pi\in
\mathbb{S}_{k-1}$. It is a possible that there exist
$\mathbf{x}\in A_{i}$ colored by $a_i$, and $\mathbf{y}\in A_j$
colored by $a_j$, such that the edge $\mathbf{x}\mathbf{y}$
belongs to $G^\pi_{k,q}$. In terms of hypergraphs (see Subsection
2.2), the hyperedges $F(\mathbf{x},\pi)$ and $F(\mathbf{y},\pi)$
share a common vertex. To avoid such conflicts, we impose some
additional conditions on $\pi$.

\begin{defn}\label{D:goodset}
We say that \( C = \{a_1, a_2, \ldots, a_t\} \subset [k] \) is a
\textit{good set} for the permutation \( \pi \in \mathbb{S}_{k-1}
\) if it satisfies the following conditions:
\begin{itemize}
    \item[(a)] \(C\) is an independent set of vertices
    in the graph \(G_\pi\).
    \item[(b)] for any pair of elements $a_i, a_j \in C$ with
    $1<a_i<a_j<k$,
     the pairs \((a_i - 1, a_j)\) and \((a_j - 1, a_i)\)
     are interlaced in \(\pi\), meaning:
    \begin{itemize}
        \item[(i)] Between \(a_i - 1\) and \(a_j\),
         either \(a_i\) or \(a_j - 1\) appears, and
        \item[(ii)] Between \(a_j - 1\) and \(a_i\),
        either \(a_i - 1\) or \(a_j\) appears.
    \end{itemize}
\end{itemize}

If $a_t=1$, then for each $a_j\in C$, $a_j$ must appear in $\pi$
between $1$ and $a_{j}$-1. Similarly, if $a_t=k$, then for each
$a_j\in C$, $a_{j}-1$ must appear in $\pi$ between $a_j$ and
$k-1$.\end{defn} Note that $1$ and $k$ can not both belong to $C$
simultaneously.

\begin{prop}\label{P:whendistanceisgood} If $C=\{a_1,a_2,
\ldots,a_{t}\}\subset [k]$ is a good set for a permutation $\pi$,
the distance coloring of $W^\pi_{k,q}$ is a well-defined Sperner
labeling.
\end{prop}
\begin{proof}
Assume that $\mathbf{x}\in A_i$ and $\mathbf{y}\in A_j$
($\mathbf{x}$, $\mathbf{y}$ are colored by $a_i$ and $a_j$,
respectively). From the definition of the distance coloring we
know that
$$x_{a_i}-x_{a_i-1}-\varepsilon_{a_i}>
x_{a_j}-x_{a_j-1}-\varepsilon_{a_j}\geq 0,\textrm{ and}$$$$
y_{a_j}-y_{a_j-1}-\varepsilon_{a_j}>
y_{a_i}-y_{a_i-1}-\varepsilon_{a_i}\geq 0.$$ Recall that
$\mathbf{xy}$ is an edge of $G^\pi_{k,q}$ if and only if
$$\mathbf{y}=\mathbf{x}+e_{\pi_a}+
e_{\pi_{a}+1}+\cdots+e_{\pi_{b-1}}, \textrm{ for some }1\leq a<b
\leq k.$$ However, under the conditions described in Definition
\ref{D:distcol}, it is impossible for
$x_{a_j}-x_{a_j-1}-\varepsilon_{a_j}$ to increase by one, while
simultaneously $x_{a_i}-x_{a_i-1}-\varepsilon_{a_i}$ decreases by
one. Therefore, $\mathbf{x}\mathbf{y}$ cannot be an edge of
$G^\pi_{k,q}$, whenever $\mathbf{x}$ and $\mathbf{y}$ are assigned
different colors by the distance coloring.
\end{proof}

Some permutations admit a large good set. For
$k=2n$ and $\pi=24\ldots(2n-2)(2n-1)(2n-3)\ldots 3 1$, the set
$C=\{1,3,\ldots,2n-1\}$ is good for $\pi$. Therefore, there exists
a distance coloring of $W^\pi_{k,q}$ with $n=\frac k 2$ colors.

The next theorem characterizes all permutations that admit a good
two-element set of colors.
\begin{thm}\label{T:imadvadobra} A permutation
$\pi\in \mathbb{S}_{k-1}$ has a good two-element set of colors if
and only if $\pi$ is not $r$-invariant.
\end{thm}
\begin{proof}
It follows from Remark \ref{R:Blocks} that an $r$-invariant
permutation does not admit a good $2$-element set of colors. Let
$\pi=\pi_1\pi_2\ldots \pi_{k-1}\in \mathbb{S}_{k-1}$ be a
permutation that does not admit $\{a,b\}$ as a good set of colors,
for any choice of $a$ and $b$. Assume that $\pi$ has $m$ blocks,
i.e., $S(\pi)=\{a_1,a_2,\ldots,a_{m-1}\}$. The blocks of $\pi$ are
determined by $S(\pi)$, see Remark \ref{R:blocksinpi}.
\begin{itemize}
   \item[(1)] For all $i>1$ the block $B_i$ begins before $B_1$
   ; that is $a_{i}+1$ appears before $1$ in $\pi$.
    Otherwise, if $a_{i}+1$ appears between $1$ and $a_i$, then
    $\{1,a_{i}+1\}$ is a good set for $\pi$.
    \item[(2)] If a block $B_i$ (with $i>1$) has more than one element,
     then $1$ lies
    between $a_{i}+1$ and $a_{i}+2$ (the first and second
    elements of $B_i$).
    Otherwise, $\{1,a_{i}+2\}$ is a good set for $\pi$.
\end{itemize}
In a similar way, considering $k$ as a possible color, we conclude
that
 \begin{itemize}
    \item[(3)] For all $i<m$ the block $B_i$ ends after $B_m$,
   that is $k-1$ appears before all $a_i\in
    S(\pi)$.
     \item[(4)] If a block $B_i$ (for $i<m$) has more than one element,
      then $k-1$ lies
    between $a_{i}-1$ and $a_{i}$ (the last two
    elements of $B_i$).
\end{itemize}
From $(1)-(4)$ it follows that
     \begin{itemize}
    \item[(5)] There are no two elements of a block $B_i$ appearing
    before the
    beginning of any other block $B_j$.
    \item[(6)] There are no two elements of a block $B_i$ appearing
    after the end
    of any other block $B_j$.
\end{itemize}
Since $\{a,b\}$ is not a good set for the permutation $\pi$, it
follows that $\pi$ does not contain a subsequence of the form
$$\cdots (a-1)\cdots (b-1)\cdots b \cdots a \cdots \textrm{ or }
\cdots (b-1)\cdots (a-1) \cdots a \cdots b\cdots. $$ Therefore, we
conclude the following
  \begin{itemize}
    \item[(7)] There are no two elements of a block $B_i$ that appear
    between
    two consecutive elements of another block $B_j$.
    \end{itemize}
    From $(1)-(7)$ we conclude that the blocks of $\pi$ are almost
    equicardinal.
    Therefore, if $k=tm+s$, then $s-1$ blocks of $\pi$ have $t+1$
elements, and $m-s+1$ blocks have $t$ elements. Furthermore, from
$(1)$ and $(4)$, we know that $\pi_m=1$ and $\pi_{s}=k-t$.

Since $\{a,b\}$ is not a good set for the permutation $\pi$, it
follows that $\pi$ does not contain a subsequence of the form
$$\cdots a\cdots b\cdots (b-1) \cdots (a-1)\cdots \textrm{ or}
\cdots b\cdots a\cdots (a-1) \cdots (b-1)\cdots. $$ Therefore, we
conclude the following
\begin{itemize}
    \item[(8)] If a block $B_i$ ends after $B_j$, then $B_{i+1}$ begins
    after
    $B_{j+1}$. Similarly, if $B_i$ ends before $B_j$,
    then $B_{i+1}$ begins
before
    $B_{j+1}$.
\end{itemize}

Note that the block $B_1$ ends before exactly $s-1$ blocks
(namely, all blocks of length $t+1$). From $(8)$, it follows that
the block $B_2$ begins before exactly $s-1$ blocks. Therefore, we
deduce that $\pi_{m-s}=t+1$. Continuing in the same manner, we
conclude that $\pi$ is the $r$-invariant permutation described in
Remark \ref{R:Blocks}.

\end{proof}

\subsection{Counting non-colored hyperedges in distance colorings}
We now determine the number of vertices of $W^\pi_{k,q}$ that are
colored with $0$ in the distance coloring described in Definition
\ref{D:distcol}. Observe that an integer point
$\mathbf{d}=(d_1,d_2,\ldots,d_t)$ of $L_{t,q}=D(W^\pi_{k,q})$ can
be interpreted as a weak composition of $m=d_1+d_2+\cdots+d_t$
into $t$ summands.
\begin{lem}\label{L:inverzzaD} Let
$\mathbf{d}=(d_1,d_2,\ldots,d_t)\in L_{t,q}$ be an integer point
such that $d_1+d_2+\cdots+d_t=m$. Then
$$|D^{-1}(\mathbf{d})|={k+q-|S(\pi)|-t-2-m\choose k-t-1}.$$
\end{lem}
\begin{proof}Note that for each $\mathbf{x}\in D^{-1}(\mathbf{d})$, the
differences
$$x_{a_i}-x_{a_i-1}=d_i+\varepsilon_{a_i}\textrm{, for all }i \in [t]$$
are fixed. Therefore, the coordinates
$x_{a_1},x_{a_2},\ldots,x_{a_t}$ are uniquely determined by their
predecessors. By the same reasoning as in Proposition
\ref{P:numberofhyperedges}, there are $q-|S(\pi)|-m$ possibile
values to assign to the remaining $k-t-1$ weakly increasing
coordinates.
\end{proof}
\noindent The preceding lemma provides the number of vertices of
$W^\pi_{k,q}$ that map to a given point $\mathbf{d}$. Importantly,
this number depends only on the sum $m$, and not on the particular
weak composition itself. A vertex $\mathbf{x}\in W^\pi_{k,q}$ is
non-colored under the distance coloring if and only if the image
$D(\mathbf{x})$ has at least two maximal coordinates.
Equivalently, these vertices correspond to weak compositions with
at least two maximal summands. Therefore, for each
$m=0,1,\ldots,q-|S(\pi)|-1$, we must count the number of weak
compositions of $m$ into $t$ parts that contain at least two
maximal summands.

For $n,t,a,b\in \mathbb{N}$, let $C(n,t,a,b)$ denote the number of
compositions of $n$ into $t$ summands each lying between $a$ and
$b$. These numbers are introduced and studied in \cite{slo}.

\begin{prop}
The number of weak compositions of an integer $m$ into $t$
summands with at least two maximal summands, denoted by $X^t_m$,
is given by
$$X^t_m=\sum_{d=\lceil\frac m t \rceil}^{\lfloor \frac m 2\rfloor}
\sum_{i=2}^{\lfloor\frac m d \rfloor}{t \choose
i}C(m+t-i(d+1),t-i,1,d).$$
\end{prop}
\begin{proof}
The maximal summand (of which there are at least two) in the
considered weak compositions of $m$ lies between $\lceil\frac m t
\rceil$ and $\lfloor \frac m 2\rfloor$. A fixed maximal summand
$d$ can appear at most $\lfloor\frac m d \rfloor$ times. If the
maximal summand $d$ appears exactly $i$ times, then the positions
of these summands can be chosen in ${t\choose i}$ ways. The
remaining $t-i$ summands, each strictly less than $d$, form a weak
composition of $m-id$. These weak compositions are in bijection
(after increasing each summand by $1$) with compositions of
$m+t-i(d+1)$ into $t-i$ parts, each lying between $1$ and $d$.
\end{proof} The following theorem is an immediate consequence of Proposition
\ref{P:whendistanceisgood} and Lemma \ref{L:inverzzaD}.

\begin{thm}\label{T:brneobojenih} Assume that $\pi\in \mathbb{S}_{k-1}$
admit a good set with $t$ elements. Then the number of non-colored
vertices in $W^\pi_{k,q}$ under the distance coloring with $t$
colors is
$$\sum_{d=0}^{q-|S(\pi)|-1}X^t_d{k+q-|S(\pi)|-t-2-d\choose k-t-1}$$
\end{thm}
From the above theorem, we conclude that the number of non-colored
vertices depends only on the parameters $k,q,t$ and $|S(\pi)|$.
For the case $t=2$, the values of $X^t_m$ simplify to
\begin{equation}\label{E:Xzat=2}
    X^2_m=\left\{%
\begin{array}{ll}
    0, & \hbox{if $m$ is odd;} \\
    1, & \hbox{if $m$ is even.} \\
\end{array}%
\right.
\end{equation}
Therefore, we can determine the exact number $N(k,q,s)$ of
non-colored vertices in the distance coloring with two colors of
$W^\pi_{k,q}$, for all permutations with $s$ adjacent inversions.

Combining Theorem \ref{T:brneobojenih} and equation
(\ref{E:Xzat=2}), it follows that
$$N(k,q,s)=\sum_{m=0}^{\left\lfloor\frac{q-s-1}{2}\right\rfloor}
{k+q-s-4-2m\choose k-3}\textrm{, which implies }$$
\begin{equation}\label{E:brneobojenihRR}
N(k,q,s)=N(k-1,q,s)+N(k,q-1,s).
\end{equation}
The ''boundary values'' for $N(k,q,s)$ are given by
$$N(k,s+2,s)= k-2, \textrm{ and }
N(s+3,q,s)=\sum_{m=0}^{\left\lfloor\frac{q-s-1}{2}\right\rfloor}
{q-1-2m\choose s} .$$
\begin{thm}  If
$\pi\in \mathbb{S}_{k-1}$ is not an $r$-invariant permutation,
then there exists a coloring of vertices of $H^{\pi}_{k,q}$ that
produces more monochromatic hyperedges than the greedy coloring.
\end{thm}
\begin{proof} The number of
non-colored hyperedges of $H^{\pi}_{k,q}$ in the greedy coloring
is ${k+q-|S(\pi)|-3\choose k-2}$, see Proposition \ref{greedy}. By
Theorem \ref{T:imadvadobra} we know that there exists a pair of
colors $\{i,j\}$ that is good for $\pi$. Applying induction and
the recurrence relation (\ref{E:brneobojenihRR}), we conclude that
the number of non-colored vertices in the distance coloring with
two colors, $N(k,q,|S(\pi)|)$, is strictly less than in the
greedy coloring case.
\end{proof}

\subsection{The optimal distance coloring}
\begin{prop} Assume that $\{i,j\}$ is a good set for a
permutation $\pi\in \mathbb{S}_{k-1}$. Then, the distance coloring
of $W^\pi_{k,q}$ with $i$ and $j$ produces at least as many
monochromatic hyperedges in $H^\pi_{k,q }$ as any other coloring
by two colors.
\end{prop}
\begin{proof}
Assume that $\ell :W^\pi_{k,q}\rightarrow \{a,b,0\}$ is a Sperner's
labeling of $W^\pi_{k,q}$ using two colors $a$ and $b$. Define the sets
$$X=\{\mathbf{x}\in W^\pi_{k,q} :\ell(\mathbf{x})=a\} \textrm{
and } Y=\{\mathbf{x}\in W^\pi_{k,q} :\ell(\mathbf{x})=b\},$$
consisting of the vertices labeled $a$ and $b$, respectively.
Further, we let\\ $A=\{\mathbf{x}\in W^\pi_{k,q}
:x_a-x_{a-1}-\varepsilon_a
>x_b-x_{b-1}-\varepsilon_b\}$ and $B=\{\mathbf{x}\in W^\pi_{k,q}
:x_a-x_{a-1}-\varepsilon_a <x_b-x_{b-1}-\varepsilon_b\}$. Define
the map $F:X\cup Y \rightarrow A\cup B$ by
$$F(\mathbf{v})=\left\{%
\begin{array}{ll}
    \mathbf{v}, & \hbox{if $\mathbf{v}\in A\cap X$
    or $\mathbf{v}\in B\cap Y$;} \\
    \mathbf{v}-e_a, & \hbox{if $\mathbf{v}\in X\setminus A$;} \\
     \mathbf{v}-e_b, & \hbox{if $\mathbf{v}\in Y\setminus B$.} \\
\end{array}%
\right.    $$ It is straightforward to verify that $F$ is an
injection. Hence, the number of colored hyperedges under the
labeling $\ell$ is at most $|A|+|B|$. Recall that $A_i$ and $A_j$
denote the sets of vertices colored by $i$ and $j$ in the distance
coloring with $i$ and $j$, see Definition \ref{D:distcol}. By the symmetry of the
simplex $R^\pi_{k,q}$, we have that $|A|+|B|=|A_i|+|A_j|$.
\end{proof}
For certain permutations, the distance coloring with two colors is
an optimal Sperner's labeling of $H^\pi_{k,q}$, maximizing the
number of monochromatic edges.
\begin{thm}For $i\in[k-1]$ such that $i\ne 1, k-1$,
and  $$\pi=(i-1)\, (i-2)\ldots 1\,  (i+1)\, (i+2)\ldots (k-1)\,
i\in \mathbb{S}_{k-1},$$ the distance coloring of $W^{\pi}_{k,q}$
with colors $1$ and $i+1$ maximizes the number of monochromatic
hyperedges among all Sperner labelings of $H^\pi_{k,q}$.
\end{thm}
\begin{proof}
Note that $\{1,i+1\}$ form an independent set of vertices in
$G_\pi$, and the conditions of Definition \ref{D:distcol} are
satisfied. In the distance coloring of $G^{\pi}_{k,q}$ the vertices
colored by $1$ and $i+1$ are respectively
$$A_1=\{\mathbf{v}\in W_{k,q}^\pi:v_1>v_{i+1}-v_i-1\}\textrm{,
}A_{i+1}=\{\mathbf{v}\in W_{k,q}^\pi:v_{i+1}-v_i-1>v_1\}.$$
 For
each vertex $\mathbf{x}\in A_1$, consider the clique $C^1_\mathbf{x}$, i.e.,
the subgraph of $G^{\pi}_{k,q}$ induced by the vertices
$$W^\pi_{k,q}\cap\{\mathbf{x},\mathbf{x}+e_{i+1},\mathbf{x}-e_1,
\mathbf{x}-e_1-e_2,\ldots,\mathbf{x}-e_1-e_2-\cdots-e_{i-1}\} .$$
Note that $C_\mathbf{x}^1$ is indeed a clique, see (\ref{ivica}).
 Similarly, for
each vertex $\mathbf{y}\in A_{i+1}$, consider the clique
$C^{i+1}_\mathbf{y}$, induced by
$$W^\pi_{k,q}\cap \{\mathbf{y},\mathbf{y}+e_1,\mathbf{y}-e_{i+1},
\mathbf{y}-e_{i+1}-e_{i+2}, \ldots,
\mathbf{y}-e_{i+1}-e_{i+2}-\cdots-e_{k-1}\} .$$

Consider any Sperner's labeling \(\ell : W^\pi_{k,q} \to [k] \cup
\{0\}\). We analyze the distribution of colored vertices among the
cliques:

\begin{itemize}
\item If \(\mathbf{v} \in A_1\), then \(\mathbf{v}\) belongs to
\begin{itemize}
    \item \(C^1_\mathbf{v}\), if \(\ell(\mathbf{v}) \in \{1, i+2, \ldots, k\}\);
    \item \(C^1_{\mathbf{v} - e_{i+1}}\), if \(\ell(\mathbf{v}) = i+1\);
    \item \(C^1_{\mathbf{v} + e_1 + \cdots + e_{j-1}}\), if \(\ell(\mathbf{v}) = j \in \{2, \ldots, i\}\).
\end{itemize}

\item If \(\mathbf{v} \in A_{i+1}\), then \(\mathbf{v}\) belongs
to
\begin{itemize}
    \item \(C^{i+1}_\mathbf{v}\), if \(\ell(\mathbf{v}) \in \{2, \ldots, i+1\}\);
    \item \(C^{i+1}_{\mathbf{v} - e_1}\), if \(\ell(\mathbf{v}) = 1\);
    \item \(C^{i+1}_{\mathbf{v} + e_{i+1} + \cdots + e_{j-1}}\), if \(\ell(\mathbf{v}) = j \in \{i+2, \ldots, k\}\).
\end{itemize}

\item If \(\mathbf{v} \in W^\pi_{k,q} \setminus (A_1 \cup
A_{i+1})\), then \(\mathbf{v}\) belongs to
\begin{itemize}
    \item \(C^{i+1}_{\mathbf{v} - e_1}\), if \(\ell(\mathbf{v}) = 1\);
    \item \(C^{1}_{\mathbf{v} + e_1 + \cdots + e_{j-1}}\), if \(\ell(\mathbf{v}) = j \in \{2, \ldots, i\}\);
    \item \(C^{1}_{\mathbf{v} - e_{i+1}}\), if \(\ell(\mathbf{v}) = i+1\);
    \item \(C^{i+1}_{\mathbf{v} + e_{i+1} + \cdots + e_{t-1}}\), if \(\ell(\mathbf{v}) = t \in \{i+2, \ldots, k\}\).
\end{itemize}
\end{itemize}

All of these cliques are well-defined and together cover the
entire set \(W^\pi_{k,q}\), accounting for every possible color
assignment to each vertex.

Moreover, since the families \(\{ C^1_\mathbf{x} \}_{\mathbf{x} \in A_1}\) and \(\{
C^{i+1}_\mathbf{y} \}_{\mathbf{y} \in A_{i+1}}\) comprise all such cliques, any
Sperner labeling \(\ell\) can assign colors to at most one vertex
within each clique. Hence, the maximum number of monochromatic
hyperedges in $H^\pi_{k,q}$ is $|A_1|+|A_{i+1}|$, which is
attained by the distance coloring with the colors $1$ and $i+1$.
\end{proof}
The above theorem resolves Problem $1a$ for the permutation
$$\pi=(i-1)\, (i-2)\ldots 1\,  (i+1)\, (i+2)\ldots (k-1)\, i.$$ By
Theorem \ref{T:izopriD_k}, it follows that a distance coloring
with two colors is optimal for all permutations obtained from
$\pi$ under the action of the dihedral group $D_k$.
\begin{cor}
For the permutation $\pi=1\, 3\, 4\, \ldots (k-1)\, 2$, the
maximum number of monochromatic hyperedges in $H^\pi_{k,q}$
achievable by any Sperner labeling is
$${k+q-3\choose k-1}- N(k,q,1).$$
\end{cor}

\end{document}